\documentclass[12pt]{amsart}
\usepackage{amsmath}
\usepackage{amsfonts}
\usepackage{amssymb}
\usepackage[all]{xy}           

\usepackage{bbding}
\usepackage{txfonts}
\usepackage[shortlabels]{enumitem}
\usepackage{ifpdf}
\ifpdf
  \usepackage[colorlinks,final,backref=page,hyperindex]{hyperref}
\else
  \usepackage[colorlinks,final,backref=page,hyperindex,hypertex]{hyperref}
\fi
\usepackage{tikz}

\makeatletter

\newtheorem{theorem}{Theorem}[section]
\newtheorem{prop}[theorem]{Proposition}
\newtheorem{defn}[theorem]{Definition}
\newtheorem{lemma}[theorem]{Lemma}
\newtheorem{coro}[theorem]{Corollary}
\newtheorem{prop-def}{Proposition-Definition}[section]

\newtheorem{exam}[theorem]{Example}

\newcommand{\nc}{\newcommand}

\nc{\delete}[1]{{}}
\nc{\mmargin}[1]{}

\nc{\mlabel}[1]{\label{#1}}  
\nc{\mcite}[2][]{\cite[#1]{#2}}  
\nc{\mref}[1]{\ref{#1}}  
\nc{\mbibitem}[1]{\bibitem{#1}} 

\delete{
\nc{\mlabel}[1]{\label{#1}  
{\hfill \hspace{1cm}{\bf{{\ }\hfill(#1)}}}}
\nc{\mcite}[1]{\cite{#1}{{\bf{{\ }(#1)}}}}  
\nc{\mref}[1]{\ref{#1}{{\bf{{\ }(#1)}}}}  
\nc{\mbibitem}[1]{\bibitem[\bf #1]{#1}} 
}

\nc{\difforg}{U_D}
\nc{\rbforg}{U_{RB}}
\nc{\multforg}{U_{M}}
\nc{\diffree}{{}_DF}
\nc{\difcof}{F_D}
\nc{\rbfree}{{}_{RB}F}
\nc{\rbcof}{F_{RB}}
\nc{\multfree}{{}_{M}F}
\nc{\multcof}{F_M}

\nc{\bin}[2]{ (_{\stackrel{\scs{#1}}{\scs{#2}}})}  
\nc{\binc}[2]{(\!\! \begin{array}{c} \scs{#1}\\
    \scs{#2} \end{array}\!\!)}  
\nc{\bincc}[2]{  ( {\scs{#1} \atop
    \vspace{-1cm}\scs{#2}} )}  
\nc{\bs}{\bar{S}}
\nc{\la}{\longrightarrow}
\nc{\ot}{\otimes}
\nc{\rar}{\rightarrow}
\nc{\dar}{\downarrow}
\nc{\dap}[1]{\downarrow \rlap{$\scriptstyle{#1}$}}
\nc{\defeq}{\stackrel{\rm def}{=}}
\nc{\dis}[1]{\displaystyle{#1}}
\nc{\dotcup}{\ \displaystyle{\bigcup^\bullet}\ }
\nc{\hcm}{\ \hat{,}\ }
\nc{\hts}{\hat{\otimes}}
\nc{\hcirc}{\hat{\circ}}
\nc{\lleft}{[}
\nc{\lright}{]}
\nc{\curlyl}{\left \{ \begin{array}{c} {} \\ {} \end{array}
    \right .  \!\!\!\!\!\!\!}
\nc{\curlyr}{ \!\!\!\!\!\!\!
    \left . \begin{array}{c} {} \\ {} \end{array}
    \right \} }
\nc{\longmid}{\left | \begin{array}{c} {} \\ {} \end{array}
    \right . \!\!\!\!\!\!\!}
\nc{\ora}[1]{\stackrel{#1}{\rar}}
\nc{\ola}[1]{\stackrel{#1}{\la}}
\nc{\scs}[1]{\scriptstyle{#1}} \nc{\mrm}[1]{{\rm #1}}
\nc{\dirlim}{\displaystyle{\lim_{\longrightarrow}}\,}
\nc{\invlim}{\displaystyle{\lim_{\longleftarrow}}\,}
\nc{\dislim}[1]{\displaystyle{\lim_{#1}}} \nc{\colim}{\mrm{colim}}
\nc{\mvp}{\vspace{0.3cm}} \nc{\tk}{^{(k)}} \nc{\tp}{^\prime}
\nc{\ttp}{^{\prime\prime}} \nc{\svp}{\vspace{2cm}}
\nc{\vp}{\vspace{8cm}}
\nc{\modg}[1]{\!<\!\!{#1}\!\!>}
\nc{\intg}[1]{F_C(#1)}
\nc{\lmodg}{\!<\!\!}
\nc{\rmodg}{\!\!>\!}
\nc{\cpi}{\widehat{\Pi}}
\nc{\sha}{{\mbox{\cyr X}}}  
\nc{\ssha}{{\mbox{\cyrs X}}} 
\nc{\tsha}{{\mbox{\cyrt X}}}
\nc{\shpr}{\diamond}    
\nc{\labs}{\mid\!}
\nc{\rabs}{\!\mid}

\font\cyr=wncyr10
\font\cyrs=wncyr7
\font\cyrt=wncyr5

\nc{\ann}{\mrm{ann}}
\nc{\Aut}{\mrm{Aut}}
\nc{\can}{\mrm{can}}
\nc{\Cont}{\mrm{Cont}}
\nc{\rchar}{\mrm{char}}
\nc{\cok}{\mrm{coker}}
\nc{\dtf}{{R-{\rm tf}}}
\nc{\dtor}{{R-{\rm tor}}}

\nc{\Div}{{\mrm Div}}
\nc{\End}{\mrm{End}}
\nc{\Ext}{\mrm{Ext}}
\nc{\Fil}{\mrm{Fil}}
\nc{\Fr}{\mrm{Fr}}
\nc{\Frob}{\mrm{Frob}}
\nc{\Gal}{\mrm{Gal}}
\nc{\GL}{\mrm{GL}}
\nc{\Hom}{\mrm{Hom}}
\nc{\hsr}{\mrm{H}}
\nc{\hpol}{\mrm{HP}}
\nc{\id}{\mrm{id}}
\nc{\im}{\mrm{im}}
\nc{\incl}{\mrm{incl}}
\nc{\length}{\mrm{length}}
\nc{\mchar}{\rm char}
\nc{\mpart}{\mrm{part}}
\nc{\ql}{{\QQ_\ell}}
\nc{\qp}{{\QQ_p}}
\nc{\rank}{\mrm{rank}}
\nc{\rcot}{\mrm{cot}}
\nc{\rdef}{\mrm{def}}
\nc{\rdiv}{{\rm div}}
\nc{\rtf}{{\rm tf}}
\nc{\rtor}{{\rm tor}}
\nc{\res}{\mrm{res}}
\nc{\SL}{\mrm{SL}}
\nc{\Spec}{\mrm{Spec}}
\nc{\tor}{\mrm{tor}}
\nc{\Tr}{\mrm{Tr}}
\nc{\tr}{\mrm{tr}}

\nc{\ab}{\mathbf{Ab}}
\nc{\DRB}{\mathbf{DRB}}
\nc{\Alg}{{\mathbf{Alg}}}
\nc{\ALG}{{\mathbf{ALG}}}
\nc{\RB}{\mathbf{RB}}
\nc{\RBA}{\mathbf{RBA}}
\nc{\bfk}{{\bf k}}
\nc{\bfone}{{\bf 1}}
\nc{\detail}{\marginpar{\bf More detail}
    \noindent{\bf Need more detail!}
    \svp}
\nc{\Diff}{\mathbf{Diff}}
\nc{\gap}{\marginpar{\bf Incomplete}\noindent{\bf Incomplete!!}
    \svp}
\nc{\FMod}{\mathbf{FMod}}
\nc{\Int}{\mathbf{Int}}
\nc{\Mon}{\mathbf{Mon}}
\nc{\Mult}{{\mathbf{Mlt}}}
\nc{\remarks}{\noindent{\bf Remarks: }}
\nc{\Rep}{\mathbf{Rep}}
\nc{\Rings}{\mathbf{Rings}}
\nc{\Sets}{\mathbf{Sets}}

\nc{\BA}{{\mathbb A}}
\nc{\CC}{{\mathbb C}}
\nc{\DD}{{\mathbb D}}
\nc{\EE}{{\mathbb E}}
\nc{\FF}{{\mathbb F}}
\nc{\GG}{{\mathbb G}}
\nc{\HH}{{\mathbb H}}
\nc{\LL}{{\mathbb L}}
\nc{\NN}{{\mathbb N}}
\nc{\PP}{{\mathbb P}}
\nc{\QQ}{{\mathbb Q}}
\nc{\RR}{{\mathbb R}}
\nc{\TT}{{\mathbb T}}
\nc{\VV}{{\mathbb V}}
\nc{\ZZ}{{\mathbb Z}}

\nc{\cala}{{\mathcal A}}
\nc{\calc}{{\mathcal C}}
\nc{\cald}{\mathcal{D}}
\nc{\cale}{{\mathcal E}}
\nc{\calf}{{\mathcal F}}
\nc{\calg}{{\mathcal G}}
\nc{\calh}{{\mathcal H}}
\nc{\cali}{{\mathcal I}}
\nc{\call}{{\mathcal L}}
\nc{\calm}{{\mathcal M}}
\nc{\caln}{{\mathcal N}}
\nc{\calo}{{\mathcal O}}
\nc{\calp}{{\mathcal P}}
\nc{\calr}{{\mathcal R}}
\nc{\cals}{{\mathcal S}}
\nc{\calt}{{\mathcal T}}
\nc{\calw}{{\mathcal W}}
\nc{\calx}{{\mathcal X}}
\nc{\CA}{\mathcal{A}}

\nc{\fraka}{{\mathfrak a}}
\nc{\frakb}{\mathfrak{b}}
\nc{\frakB}{{\frak B}} \nc{\frakm}{{\frak
m}} \nc{\frakM}{{\frak M}}
\nc{\frakp}{{\frak p}}
\nc{\frakS}{{\frak S}} \nc{\frakA}{{\frak A}} \nc{\frakx}{{\frak
x}}

\nc{\Dif}{\mathbf{Dif}}
\nc{\DIF}{\mathbf{DIF}}
\nc{\G}{\mathbf{G}}
\nc{\C}{\mathbf{C}}
\nc{\A}{\mathbf{A}}
\nc{\B}{\mathbf{B}}
\nc{\T}{\mathbf{T}}

\begin{document}
\title[Monads and distributive laws]{
Monads and distributive laws for Rota-Baxter and differential algebras}

\author{Li Guo}
\address{
Department of Mathematics and Computer Science,
Rutgers University,
Newark, NJ 07102, USA}
\email{liguo@newark.rutgers.edu}

\author{William Keigher}
\address{
Department of Mathematics and Computer Science,
Rutgers University,
Newark, NJ 07102, USA}
\email{keigher@newark.rutgers.edu}

\author{Shilong Zhang}
\address{Department of Mathematics, Lanzhou University, Lanzhou, Gansu, 730000, China}
\email{2663067567@qq.com}

\date{\today}

\begin{abstract}
In recent years, algebraic studies of the differential calculus and integral calculus in the forms of differential algebra and Rota-Baxter algebra have been merged together to reflect the close relationship between the two calculi through the First Fundamental Theorem of Calculus.
In this paper we study this relationship from a categorical point of view in the context of distributive laws which can be tracked back to the distributive law of multiplication over addition. The monad giving Rota-Baxter algebras and the comonad giving differential algebras are constructed. Then a mixed distributive law of the monad over the comonad is established. As a consequence, we obtain monads and comonads giving the composite structures of differential and Rota-Baxter algebras.
\end{abstract}


\keywords{
differential algebra, Rota-Baxter algebra, differential Rota-Baxter algebra, monad, comonad, distributive law
}

\maketitle

\tableofcontents

\setcounter{section}{0}

\allowdisplaybreaks

\section{Introduction}
In this paper, we study a monad giving Rota-Baxter algebras and a comonad giving differential algebras, and a mixed distributive law of this monad over this comonad.  We also give significant consequences of this mixed distributive law.

For the last few decades, differential algebra, first of weight zero and then of weight $\lambda$, and its integral counter part, Rota-Baxter algebra of weight $\lambda$, have attracted much attention, with broad applications. Started as an algebraic study of differential equations, differential algebra has become the theoretical foundation of W.-T. Wu's ground breaking work on mechanical proofs of geometric theorems~\mcite{Wu2}. It is also closely related to symbolic computation~\mcite{RR,RRTB,Hoe}. Originated in the probabilistic study of G. Baxter in 1960~\mcite{Ba}, Rota-Baxter algebra is the algebraic abstraction of integration and has found surprising applications in integrable systems and quantum field theory, being one of the two fundamental algebraic structures (the other one being Hopf algebra) in the algebraic approach of Connes-Kreimer on renormalization of quantum field theory~\mcite{Bai,C-K1,EGM,GZ}. See~\mcite{Gub,Rot3,RS} for general discussions on Rota-Baxter algebras. Recently, the coupling of Rota-Baxter and differential algebras has been made apparent as differential Rota-Baxter algebra~\mcite{GK3} and as integro-differential algebra~\mcite{GRR,RR}, both motivated by the first fundamental theorem of calculus and the integration by parts formula. The latter, being a stronger version of the former, has its motivation and applications in algebraic and computational study of boundary value problems of ordinary and partial differential equations~\mcite{RRTB}.

In this paper, we justify the interaction of the Rota-Baxter (integral) operator and the differential operator from a categorical point of view by providing a mixed distributive law between the monad and comonad giving these two operators.

The concept of a distributive law between monads was introduced by Beck in his seminal work~\cite{Be} as a categorical generalization of the distributive law of multiplication over addition. In general, a distributive law relates two monads arising from free constructions. Since Beck's work, distributive laws have been studied in many contexts and generalities, such as module theory, category theory, higher category theory, combinatorics, operads, Yang-Baxter equations, quantum groups, computational mathematics, computer science and information science~\mcite{BS,Bo,BLS,Bou,BLV,BV,Ch,DK,Ga,HM,HP,Ja,LS,MRW,RW,St,Web,Wo}. The free construction of Rota-Baxter algebras and the cofree construction of differential algebras provide a natural and non-trivial context for a mixed distributive law to be determined. This is the main purpose of the current paper. The existence of this mixed distributive law suggests a close connection between the corresponding algebraic structures.
We will also describe the significant consequences of the existence of this mixed distributive law.

In Section~\mref{sec:rba} we recall the construction of free Rota-Baxter algebras by the mixable shuffle (quasi-shuffle) product and obtain the monad giving Rota-Baxter algebras over algebras. In Section~\mref{sec:dif} we similarly recall the construction of cofree differential algebras (with a weight) and obtain the comonad giving differential algebras over algebras. Then in Section~\mref{sec:mix}, the mixed distributive law of the monad giving Rota-Baxter algebras on algebras over the comonad giving differential algebras on algebras is obtained in Theorem~\mref{thm:cddist}. Applications are given in Section~\mref{sec:diffrba} to differential Rota-Baxter algebras obtained by combining Rota-Baxter algebras and differential algebras through the first fundamental theorem of calculus.

Throughout the paper, we fix a commutative ring $\bfk$ with identity, and we fix an element $\lambda \in \bfk$.  All algebras we consider will be commutative
$\bfk$-algebras with identity, and all homomorphisms of algebras will be $\bfk$-algebra homomorphisms that preserve the identity. Likewise all linear maps and tensor products will be taken over $\bfk$. Thus references to $\bfk$ will be suppressed unless a specific $\bfk$ is emphasized or a reminder is needed.
We write $\NN$ for the additive monoid of
natural numbers $\{0,1,2,\ldots\}$ and
$\NN_+=\{ n\in \NN\mid n>0\}$ for the positive integers.
In this paper, we use the categorical notations as in ~\cite{Ma}.

\section{Rota-Baxter algebras and monads}
\mlabel{sec:rba}
In this section, we recall the concept of a Rota-Baxter algebra and the construction of free Rota-Baxter algebras by the mixable shuffle product. See~\mcite{Gub,GK1} for details. We then give the monad whose algebras are the Rota-Baxter algebras.

\subsection{Free Rota-Baxter algebras}
We begin with the definition of Rota-Baxter algebras.
\begin{defn}
Let $R$ be an algebra.
\begin{enumerate}
\item
A {\bf Rota-Baxter operator of weight $\lambda$ on $R$},
or simply a {\bf Rota-Baxter operator} on $R$
is a $\bfk$-linear endomorphism $P$ of $R$ satisfying
\begin{equation}
 P(x)P(y)=P(xP(y))+P(yP(x))+\lambda P(xy),\ \text{ for all } x,\ y\in R.
\mlabel{eq:bax1}
\end{equation}
\item
    A {\bf Rota-Baxter algebra of weight $\lambda$},
or simply a {\bf Rota-Baxter algebra},
is a pair $(R,P)$ where $R$ is an algebra
and $P$ is a Rota-Baxter operator on $R$.
\item
Let $(R,P)$ and $(S,Q)$ be two Rota-Baxter algebras.
A {\bf homomorphism of Rota-Baxter algebras}
$f:(R,P)\rar (S,Q)$ is a homomorphism
$f: R \rar S$ of algebras with the property that
$ f(P(x))=Q(f(x))$
for all $x \in R$.
\end{enumerate}
\end{defn}

\begin{exam}
Let $\Cont(\RR)$ denote the $\RR$-algebra of continuous functions on $\RR$.
For a given $\lambda \in \RR, \lambda>0$, let $R$ be an $\RR$-subalgebra of $\Cont(\RR)$ that is closed under the operators
$$P_0(f)(x)=-\int_x^\infty f(t)dt,\quad  P_\lambda(f)(x)=-\lambda\sum_{n\geq 0} f(x+n\lambda).$$
For example, $R$ can be taken to be the $\RR$-subalgebra generated by $e^{-x}$: $R=\sum_{k\geq 1} \RR e^{-kx}$.
Then $P_0$ is a Rota-Baxter operator of weight $0$, and $P_\lambda$ is a Rota-Baxter operator of weight $\lambda$.
\mlabel{ex:int}
\end{exam}

Let $\RBA_{\bfk,\lambda}$, or simply $\RBA$, denote the category
of commutative Rota-Baxter algebras of weight $\lambda$, and let $\ALG_{\bfk}$ or $\ALG$ denote the category of commutative algebras.  Let $U: \RBA \rar \ALG$
denote the forgetful functor given on objects $(R,P) \in \RBA$
by $U(R,P) = R$ and on morphisms $f:(R,P)\rar (S,Q)$ in $\RBA$
by $U(f) = f:R \rar S$.  In~\cite{GK1}, we proved that $U$ has a
left adjoint, and we gave an explicit description of the left
adjoint, the free commutative Rota-Baxter algebra functor.
See~\cite{Ca,Rot1} for earlier constructions of free commutative Rota-Baxter algebras on sets, that is, as a left adjoint of the forgetful functor from $\RBA$ to the category of sets.

We recall from~\cite{GK1} some general observations
about the free commutative Rota-Baxter algebra of weight $\lambda$ on a commutative
algebra $A$ with identity $\bfone_A$. The product for this free Rota-Baxter algebra on $A$ is constructed in terms of a generalization of the shuffle product, called the mixable shuffle product which in its recursive form is a natural generalization of the quasi-shuffle product~\cite{Ho} that we will describe below. This free commutative Rota-Baxter algebra on $A$ is denoted by $\sha(A)$. As a module, we have
$$\sha(A) = \bigoplus\limits_{i\geq 1}A^{\otimes i}=A\oplus (A\otimes A)\oplus (A\otimes A\otimes A)\oplus\cdots$$ where the tensors are defined over $\bfk$. The multiplication on $\sha(A)$ is the product $\diamond$ defined as follows. Let $\mathfrak{a}=a_0\otimes \cdots \otimes a_m\in A^{\otimes(m+1)}$ and $\mathfrak{b}=b_0\otimes \cdots \otimes b_n\in A^{\otimes(n+1)}$. If $mn=0$, define

\begin{equation}
\mathfrak{a}\diamond \mathfrak{b}=
\begin{cases}
(a_0b_0)\otimes b_1 \otimes\cdots \otimes b_n, & m=0, n>0, \\
(a_0b_0)\otimes a_1 \otimes\cdots \otimes a_m, & m>0, n=0, \\
a_0b_0,                                        & m=n=0. \\
\end{cases}
\mlabel{eq:mshprod1}
\end{equation}
If $m>0$ and $n>0$, then $\mathfrak{a}\diamond \mathfrak{b}$ is defined inductively on $m+n$ by
\begin{eqnarray}
&&(a_0b_0)\otimes\Big((a_1\otimes\cdots\otimes a_m)\diamond (\bfone_A\otimes b_1\otimes \cdots b_n)
                 +(\bfone_A\otimes a_1\otimes \cdots\otimes a_m)\diamond (b_1\otimes \cdots b_n)\notag\\
&&                 +\lambda (a_1\otimes\cdots\otimes a_m)\diamond (b_1\otimes \cdots b_n)\Big).
\mlabel{eq:mshprod2}
\end{eqnarray}

Extending by additivity, $\shpr$ gives a $\bfk$-bilinear map

\[ \shpr: \sha(A) \times \sha(A) \rar \sha(A). \]
Since the product $\shpr$ restricts to the product on $A$, we will usually suppress the symbol $\shpr$ and simply denote $x y$ for $x\shpr y$ in $\sha(A)$.

As an example of computing the product, we consider
$$
(a_0\ot a_1) (b_0\ot b_1\ot b_2) =
(a_0b_0)\ot \Big( a_1 (\bfone_A\ot b_1\ot b_2) + (\bfone_A\ot a_1) (b_1\ot b_2) +\lambda(a_1 (b_1\ot b_2))\Big).
$$
By Eq.~(\mref{eq:mshprod1}), the first and third terms in the right tensor factor are just $a_1\ot b_1\ot b_2$ and $\lambda a_1b_1\ot b_2$ respectively. For the second term, we have
$$
(\bfone_A\ot a_1) (b_1\ot b_2) = b_1 \ot \left(a_1 (\bfone_A\ot b_2) + (\bfone_A\ot a_1) b_2 +\lambda a_1 b_2\right)
= b_1\ot (a_1\ot b_2 + b_2\ot a_1 +\lambda a_1b_2).
$$
Thus we obtain
$$
(a_0\ot a_1) (b_0\ot b_1\ot b_2) = (a_0b_0)\ot (a_1\ot b_1\ot b_2 + b_1\ot a_1\ot b_2 + b_1\ot b_2\ot a_1 + \lambda b_1\ot a_1b_2 + \lambda a_1b_1\ot b_2).
$$

Define a linear endomorphism $P_A$ on
$\sha(A)$ by assigning
\[ P_A( x_0\otimes x_1\otimes \ldots \otimes x_n)
=\bfone_A\otimes x_0\otimes x_1\otimes \ldots\otimes x_n, \]
for all
$x_0\otimes x_1\otimes \ldots\otimes x_n\in A^{\otimes (n+1)}$
and extending by additivity.
Let $j_A:A\rar \sha(A)$ be the canonical inclusion map.
We proved the following theorem in~\cite{GK1}.

\begin{theorem}
\begin{enumerate}
\item
The module $\sha(A)$, together with the multiplication
$\shpr$, is an algebra which will still be denoted by
$\sha(A)$.
\item
$(\sha(A),P_A)$, together with the natural embedding
$j_A:A\rightarrow \sha(A)$,
is a free Rota-Baxter algebra on $A$.
In other words,
for any Rota-Baxter algebra $(R,P)$ and any
algebra homomorphism
$\varphi:A\rar R$, there exists
a unique Rota-Baxter algebra homomorphism
$\tilde{\varphi}:(\sha(A),P_A)\rar (R,P)$ such that
$\varphi = U(\tilde{\varphi}) \circ j_A$.
\end{enumerate}
\mlabel{thm:shua}
\end{theorem}

Let $F:\ALG \rar \RBA$ denote the functor given on
objects $A \in \ALG$ by $F(A) = (\sha(A),P_A)$ and on morphisms
$f:A \rar B$ in $\ALG$ by $$F(f)\left(\sum^{k}_{i=1}a_{i0}\otimes a_{i1}\otimes\cdots \otimes a_{in_{i}}\right)=\sum^{k}_{i=1}f(a_{i0})\otimes f(a_{i1})\otimes\cdots \otimes f(a_{in_{i}})$$
which we also denote by $\sha(f)$. As above, $U:\RBA \rightarrow \ALG$ denotes the forgetful
functor defined on objects $(R,P) \in \RBA$ by $U(R,P) = R$ and
on morphisms $f:(R,P) \rightarrow (S,Q)$ in $\RBA$ by $U(f) = f$.
Next, we define two natural transformations $\eta: \id_{\ALG} \rightarrow U F$ and $\varepsilon: F U \rightarrow \id_{\RBA}$.
For any $A \in \ALG$, we define
$$\eta_{A}:A\rightarrow (U F)(A) = \sha(A)$$
to be just the natural embedding $j_A: A \rar \sha(A)$, and for any $(A,P) \in \RBA$, define
$$\varepsilon_{(A,P)}: (F  U)(A,P) = (\sha(A),P_A) \rightarrow (A,P)$$
by
$$\varepsilon_{(A,P)}\left(\sum^{k}_{i=1}a_{i0}\otimes a_{i1}\otimes\cdots \otimes a_{in_{i}}\right)=
\sum^{k}_{i=1}a_{i0}P(a_{i1}P(\cdots P(a_{in_{i}})\cdots)),$$
for any $\sum\limits_{i=1}^{k}a_{i0}\otimes a_{i1}\otimes\cdots \otimes a_{in_{i}} \in \sha(A).$

From a general principle of category theory~\mcite{Ma}, as an equivalent to Theorem~\ref{thm:shua}, we have

\begin{coro}
The functor $F:\ALG \rightarrow \RBA$ defined above is the left adjoint of the forgetful functor
$U:\RBA \rightarrow \ALG$. More precisely, there is an adjunction $\langle F, U, \eta, \varepsilon \rangle:\ALG\rightharpoonup\RBA$.
\end{coro}

\subsection{The monad giving Rota-Baxter algebras}
The above adjunction $\langle F, U, \eta, \varepsilon \rangle:\ALG\rightharpoonup\RBA$ gives rise to a monad $\mathbf{T}=\langle T,\eta,\mu\rangle$ on $\ALG$, where $T$ is the functor
$$T := U F:\ALG \rightarrow \ALG$$
and $\mu$ is the natural transformation
$$\mu:=U\varepsilon F: TT\rightarrow T.$$
Indeed, for any $A \in \ALG$, $T(A)=\sha(A)$ and
$\mu_{A}:\sha(\sha(A))\rightarrow \sha(A)$ is extended additively from
 \begin{eqnarray*}
&&\mu_{A}( (a_{00}\otimes\cdots\otimes a_{0n_{0}})\otimes\cdots\otimes(a_{k0}\otimes\cdots\otimes a_{kn_{k}}))\\
&=&(a_{00}\otimes\cdots\otimes a_{0n_{0}})  P_{A}(\cdots P_{A}(a_{k0}\otimes\cdots\otimes a_{kn_{k}})\cdots),
\end{eqnarray*}
where $(a_{00}\otimes\cdots\otimes a_{0n_{0}})\otimes\cdots\otimes(a_{k0}\otimes\cdots\otimes a_{kn_{k}})\in\sha(\sha(A))$ with
$a_{i0}\otimes\cdots\otimes a_{in_{i}}\in A^{\ot (n_{i}+1)}$ for $n_0,\ldots,n_k\geq 0$ and $0\leq i\leq k$.

By~\cite{Ma}, the monad $\mathbf{T}$ induces a category
of $\mathbf{T}$-algebras, denoted by $\ALG^{\mathbf{T}}$. The objects in $\ALG^{\mathbf{T}}$ are pairs $\langle A,h\rangle$ where $A\in \ALG$ and
$h:\sha(A)\rightarrow A$ is an algebra homomorphism satisfying
the two properties
\begin{equation}
h\circ\eta_A=\id_A,\quad\quad h\circ T(h)=h\circ \mu_{A}.
\mlabel{eq:Tcate}
\end{equation}
A morphism $\phi:\langle R,f\rangle\rightarrow\langle S,g\rangle$ in $\ALG^{\mathbf{T}}$ is an algebra homomorphism $\phi:R\rightarrow S$ such that $g\circ T(\phi)=\phi \circ f$.

The monad $\mathbf{T}$ gives rise to an adjunction
$$\langle F^{\mathbf{T}}, U^{\mathbf{T}},\eta^{\mathbf{T}} , \varepsilon^{\mathbf{T}} \rangle:\ALG
\rightharpoonup \ALG^{\mathbf{T}},$$
where
$$F^{\mathbf{T}} : \ALG \rightarrow \ALG^{\mathbf{T}}$$ is given on objects
by $F^{\mathbf{T}}(A) = \langle \sha(A),\mu_{A}\rangle$ and on homomorphisms
$\varphi:A\rightarrow B$ in $\ALG$ by $F^{\mathbf{T}}(\varphi)=F(\varphi)$.  The functor
$$U^{\mathbf{T}}:\ALG^{\mathbf{T}}\rightarrow \ALG$$
is defined on objects
$\langle A,h\rangle$ by $U^{\mathbf{T}}\langle A,h\rangle = A$, and on morphisms
$\phi:\langle R,f\rangle\rightarrow\langle S,g\rangle$ in $\ALG^{\mathbf{T}}$
by $U^{\mathbf{T}}(\phi)=\phi$.  The natural transformations $\varepsilon^{\mathbf{T}}$ and
$\eta^{\mathbf{T}}$ are defined similarly as $\varepsilon$ and $\eta$, respectively.
Then there is a uniquely defined comparison functor
$K:\RBA \rightarrow \ALG^{\mathbf{T}}$ given by $K(R, P)=\langle R,U(\varepsilon_{(R,P)})\rangle$ for any $(R, P)\in\RBA$ such that $K F=F^{\mathbf{T}}$
and $U^{\mathbf{T}} K=U$.

\begin{theorem}
The comparison functor $K:\RBA \rightarrow \ALG^{\mathbf{T}}$ is
an isomorphism, i.e., $\RBA$ is monadic over $\ALG$.
\mlabel{thm:rbmon}

\end{theorem}

\begin{proof}
We will use Beck's Theorem~\cite[Theorem 1, p.147]{Ma}
to show that $K$ is an isomorphism. Thus we only need to show that the functor
$U$ creates coequalizers for those parallel pairs $f, g:(R,P) \rightarrow (R',P')$ in $\RBA$
for which the pair $U(f), U(g):R\rightarrow R'$ has a split coequalizer in $\ALG$.
Because the pair $U(f), U(g)$ has a split coequalizer in $\ALG$, there are algebra homomorphisms
$e:R'\rightarrow \overline{R}$, $t:R'\rightarrow R$ and $s:\overline{R}\rightarrow R'$
such that $e$ is a coequalizer of the pair $U(f), U(g)$,
$e \circ s = \id_{\overline{R}}$, $U(f) \circ t =
\id_{R'}$, and $s \circ e = U(g) \circ t$.

Define a linear operator on
 $\overline{R}$ by $\overline{P}=e\circ P'\circ s$. For any $a,b\in \overline{R}$, we have
\begin{eqnarray*}
\overline{P}(a)\overline{P}(b)&=&(e\circ P'\circ s)(a)(e\circ P'\circ s)(b)\\
&=&e(P'(s(a))P'(s(b)))\\
&=&e(P'(P'(s(a))s(b)+s(a)P'(s(b)))+\lambda P'(s(a)s(b)))\\
&=&(e\circ P')(P'(s(a))s(b)+s(a)P'(s(b)))+(e\circ P'\circ s)(\lambda ab)\\
&=&(e\circ P'\circ (U(f) \circ t))(P'(s(a))s(b)+s(a)P'(s(b)))+\lambda \overline{P}(ab)\\
&=&(e\circ (U(f)\circ P) \circ t)(P'(s(a))s(b)+s(a)P'(s(b)))+\lambda \overline{P}(ab)\\
&=&((e\circ U(g))\circ P \circ t)(P'(s(a))s(b)+s(a)P'(s(b)))+\lambda \overline{P}(ab)\\
&=&(e\circ (P'\circ U(g)) \circ t)(P'(s(a))s(b)+s(a)P'(s(b)))+\lambda \overline{P}(ab)\\
&=&(e\circ P'\circ (s \circ e))(P'(s(a))s(b)+s(a)P'(s(b)))+\lambda \overline{P}(ab)\\
&=&(e\circ P'\circ s)((e\circ P'\circ s)(a)(e\circ s)(b)+(e\circ s)(a)(e\circ P'\circ s)(b))+\lambda \overline{P}(ab)\\
&=&\overline{P}(\overline{P}(a)b+a\overline{P}(b)+\lambda ab).
\end{eqnarray*}
Therefore $\overline{P}$ is a Rota-Baxter operator of weight $\lambda$ on $\overline{R}$.
Further,
\begin{eqnarray*}
\overline{P}\circ e
&=&(e\circ P'\circ s)\circ e
=(e\circ P')\circ (s\circ e)
=(e\circ P')\circ (U(g) \circ t)
=e\circ (U(g)\circ P) \circ t\\
&=&(e\circ U(f))\circ P \circ t
=e\circ (P'\circ U(f)) \circ t
=e\circ P'\circ (U(f) \circ t)
=e\circ P'.
\end{eqnarray*}
Hence $e$ is a Rota-Baxter algebra homomorphism.

It remains to show that
$e:(R',P') \rightarrow ( \overline{R},\overline{P})$ is a coequalizer
of the pair $f, g$ in $\RBA$. Suppose that  $h:(R',P') \rightarrow (R'',P'')$ is a Rota-Baxter algebra homomorphism such that $h\circ f=h\circ g.$ Hence $U(h)\circ U(f)=U(h)\circ U(g)$ holds. Then there is a unique morphism $h':\overline{R}\rightarrow R''$ in $\ALG$
such that $U(h)=h'\circ e$ holds in $\ALG$. Now
\begin{eqnarray*}
h'\circ \overline{P}
&=&h'\circ (e\circ P'\circ s)
=U(h)\circ (P'\circ s)
=U(h)\circ P'\circ(U(f)\circ t)\circ s\\
&=&U(h)\circ( U(f)\circ P )\circ t\circ s
=(U(h)\circ U(g) )\circ P \circ t\circ s\\
&=&U(h)\circ (P' \circ U(g)) \circ t\circ s
=(P''\circ U(h)) \circ U(g) \circ t\circ s \\
&=&P''\circ U(h) \circ (s \circ e)\circ s
=P''\circ U(h) \circ s
=P''\circ (h'\circ e) \circ s
=P''\circ h'.
\end{eqnarray*}
Therefore $h':(\overline{R},\overline{P})\rightarrow (R'',P'')$  is a Rota-Baxter algebra homomorphism and so $e:(R',P') \rightarrow (\overline{R},\overline{P})$ is a coequalizer
of the pair $f, g$ in $\RBA$, as desired.
\end{proof}

For any $\bfk$-module $M$, let $\mathbf{0}_{M}:M\rightarrow \bfk$ denote the zero map given by $\mathbf{0}_{M}(m)=0$ for any $m\in M$.

\begin{coro}
For any algebra $A$, there is a one-to-one correspondence
between
\begin{enumerate}
\item Rota-Baxter operators $P$ on $A$;
\mlabel{it:rb'1}
\item $\mathbf{T}$-structures on $A$, i.e., algebra homomorphisms $h:\sha(A)\rightarrow A$ satisfying both
$h\circ\eta_{A}=\id_{A}$ and $h\circ T(h)=h\circ \mu_{A}$;
\mlabel{it:rb'2}
\item
linear maps $h_n:\sha(A) \rightarrow A, n\in\NN_+,$ satisfying the following conditions.
\begin{enumerate}
\item $h_n|_{A^{\otimes i}}=\mathbf{0}_{A^{\otimes i}}$ for $i\neq n$.
\mlabel{it:3'1}
\item
$h_1|_{A}=\id_A$.
\mlabel{it:3'2}
\item
For $n_0, n_{1}, \cdots, n_k\geq 1$ with $k\in\NN$, we have
$h_{k+1}(h_{n_0}\ot h_{n_1}\ot\cdots \ot h_{n_k})=\left(\sum\limits_{i=1}^{\infty} h_i\right) \mu_{k+1}$. Here $\mu_{k+1}$ is the linear map $\sha(\sha(A))\rightarrow A$ given by
\begin{align*}
\begin{cases}
\mu_{k+1}|_{{\ssha(A)}^{\ot i}}=\mu_{A}|_{{\ssha(A)}^{\ot (k+1)}}  \quad\quad&\text{for}~~i=k+1,\\
\mu_{k+1}|_{{\ssha(A)}^{\ot i}}=\mathbf{0}_{{\ssha(A)}^{\ot i}}  \quad\quad&\text{for}~~i\neq k+1,
\end{cases}
\end{align*}
and
$$(h_{n_0}\ot \cdots \ot h_{n_k})|_{{\ssha(A)}^{\ot i}}=\mathbf{0}_{{\ssha(A)}^{\ot i}}
\quad for \quad i\neq k+1.$$
\mlabel{it:3'3}
\item
$\left(\sum\limits_{i=1}^{\infty}h_i\right)(\mathfrak{a} \mathfrak{b})=h_{m+1}(\mathfrak{a})h_{n+1}(\mathfrak{b})$
for any $\mathfrak{a}\in A^{\otimes {(m+1)}}, \mathfrak{b}\in A^{\ot (n+1)}$.
\mlabel{it:3'4}
\end{enumerate}
\mlabel{it:rb'3}
\end{enumerate}
\end{coro}

\begin{proof}
The equivalence of the first two conditions follows directly from Theorem~\ref{thm:rbmon}.

Next we give the equivalence of Item~(\mref{it:rb'2}) and Item~(\mref{it:rb'3}).

For a given $h:\sha(A)\to A$ satisfying the conditions in Item~(\mref{it:rb'2}), define linear maps $h_n:\sha(A)\to A$ by $h_n|_{A^{\ot n}}=h|_{A^{\ot n}}$ and $h_n|_{A^{\ot i}}=\mathbf{0}|_{A^{\ot i}}$ with $i\neq n$. Then we have $h_1|_A=h\circ \eta_{A}=\id_A$.
That is, Item~(\mref{it:3'1}) and Item~(\mref{it:3'2}) hold.
Note that $T(h)=\sum\limits_{k=1}^\infty h^{\ot k}$.
Next let $k\geq 0$ and $n_0, n_1,\cdots, n_k\geq 1$ be given. Let $s\geq 0$ and let $\mathfrak{w}=\mathfrak{w}_{0}\ot \mathfrak{w}_{1}\ot \cdots \ot \mathfrak{w}_{s}\in\sha(A)^{\ot (s+1)}$ with
$\mathfrak{w}_{j}\in A^{\ot m_j}$ where $m_j\geq 1$ for $0\leq j\leq s$.
If $s=k$ and $m_j=n_j$, then
$h_{k+1}((h_{n_0}\ot h_{n_1}\ot \cdots \ot h_{n_k})(\mathfrak{w}))=h( T(h)(\mathfrak{w}))$ and
$\left(\sum\limits_{i=1}^{\infty} h_i\right)( \mu_{k+1}(\mathfrak{w}))=h( \mu_{A}(\mathfrak{w}))$.
By the condition $h\circ T(h)=h\circ \mu_{A}$ in Item~(\mref{it:rb'2}), we get $h_{k+1}((h_{n_0}\ot h_{n_1}\ot \cdots \ot h_{n_k})(\mathfrak{w}))=\left(\sum\limits_{i=1}^{\infty} h_i\right)(\mu_{k+1}(\mathfrak{w}))$. Otherwise, $h_{k+1}((h_{n_0}\ot h_{n_1}\ot \cdots \ot h_{n_k})(\mathfrak{w}))=\left(\sum\limits_{i=1}^{\infty} h_i\right)( \mu_{k+1}(\mathfrak{w}))=0$.
Since $h_{n}$ and $\mu_{k+1}$ are linear maps, we obtain that Item~(\mref{it:3'3}) holds.
As $h$ is an algebra homomorphism, for any $\mathfrak{a}\in A^{\ot (m+1)}$, $\mathfrak{b}\in A^{\ot (n+1)}$, we have $h(\mathfrak{a} \mathfrak{b})=h(\mathfrak{a})h(\mathfrak{b})$.
For
$\Big(\sum\limits_{k=1}^{\infty}h_k\Big)(\mathfrak{a} \mathfrak{b})=h(\mathfrak{a} \mathfrak{b})$ and
$h_{m+1}(\mathfrak{a})h_{n+1}(\mathfrak{b})=h(\mathfrak{a})h(\mathfrak{b})$,
we get $\Big(\sum\limits_{i=1}^{\infty}h_i\Big)(\mathfrak{a} \mathfrak{b})=h_{m+1}(\mathfrak{a})h_{n+1}(\mathfrak{b})$.
That is, Item~(\mref{it:3'4}) holds.

Conversely, for given $h_n, n\geq 0$, satisfying the conditions in Item~(\mref{it:rb'3}), define $h:=\sum\limits_{n=1}^\infty h_n$. Then $h$ is a linear map on $\sha(A)=\bigoplus\limits_{k=1}^\infty A^{\ot k}$. By the linearity of $h$, Item~(\mref{it:3'4}) shows that $h$  is an algebra homomorphism. From $h_1=\id_A$, we obtain $h\circ \eta_{A}=h_1\circ \eta_{A}=\id_{A}$. Next from $h_{k+1}(h_{n_0}\ot \cdots \ot h_{n_k})=\left(\sum\limits_{i=1}^{\infty} h_i\right) \mu_{k+1}$ and the definition of $h$,
we obtain $h\circ T(h)=h\circ \mu_{A}$ in Item~(\mref{it:rb'2}) for $h$.
\end{proof}

\section{Differential algebras and comonads}
\mlabel{sec:dif}

In this section, we review background on differential algebras. We also study the comonad giving differential algebras.

Recall that a derivation $d$ on an algebra $A$ is a linear map $d:A\rightarrow A$
that satisfies Leibnitz's rule: $d(ab)=d(a)b+ad(b)$ for all $a,b\in A$. Recall now the concept of a derivation with a weight which was introduced in~\mcite{GK3} as a generalization of that of a derivation.
\begin{defn}
Let $\bfk$ be a ring, $\lambda\in \bfk$,
and let $R$ be an algebra.
\begin{enumerate}
\item
A {\bf derivation of weight $\lambda$ on $R$ over $\bfk$}
or more briefly, a {\bf $\lambda$-derivation on $R$ over $\bfk$}
is a module
endomorphism $d$ of $R$ satisfying both
\begin{equation}
 d(xy)=d(x)y+xd(y)+\lambda d(x)d(y),\ \text{ for all } x,\ y\in R
\mlabel{eq:der1}
\end{equation}
and
\begin{equation}
d(\bfone_{R})=0.
\mlabel{eq:d(1)=0}
\end{equation}
\item
    A {\bf $\lambda$-differential algebra}
is a pair $(R,d)$ where $R$ is an algebra
and $d$ is a $\lambda$-derivation
on $R$ over $\bfk$.
\item
Let $(R,d)$ and $(S,e)$ be two $\lambda$-differential
algebras.
A {\bf homomorphism of $\lambda$-differential algebras}
$f:(R,d)\rar (S,e)$ is a homomorphism
$f: R \rar S$ of algebras such that
$ f(d(x))=e(f(x))$
for all $x \in R$.
\end{enumerate}
\end{defn}

Note that if $\lambda = 0$, then a $0$-derivation is a derivation
in the usual sense~\cite{Kol}.

\begin{exam}
As an example of a $\lambda$-derivation, let $\RR$ denote
the field of real numbers, and let $\lambda \in \RR$,
$\lambda \neq 0$.  Let $A$ denote the $\RR$-algebra of
$\RR$-valued analytic functions on $\RR$, and consider the usual
"difference quotient" operator $d_\lambda$ on $A$ defined by
$$(d_\lambda(f))(x) = (f(x+\lambda) - f(x))/\lambda.$$
Then a simple calculation shows that $d_\lambda$ is a $\lambda$-derivation on $A$.  Furthermore, $d_\lambda \circ P_\lambda = {\rm id}_A$, where $P_\lambda$ is the Rota-Baxter operator from Example~\ref{ex:int}.
\mlabel{ex:difquo}
\end{exam}

The following generalization of the well-known result
of Leibnitz~\cite[p.60]{Kol} was proved in~\cite{GK3}.
\begin{prop}
Let $(R,d)$ be a $\lambda$-differential algebra, let $x,y \in
R$, and let $n \in \NN$.  Then
\begin{equation}
d^{(n)}(xy) =
\sum_{k=0}^{n}\sum_{j=0}^{n-k}\binc{n}{k}\binc{n-k}{j}
\lambda^{k}d^{(n-j)}(x)d^{(k+j)}(y).
\mlabel{eq:der2}
\end{equation}
\mlabel{prop:Leibnitz}
\end{prop}

We next recall the concept and basic properties of the algebra of $\lambda$-Hurwitz series~\cite{GK3} as a generalization of the ring of Hurwitz series~\cite{Ke}.
For any algebra $A$, let $A^{\NN}$ denote the $\bfk$-module of all functions $f:\NN \rightarrow A$.  On $A^{\NN}$, we define the {\bf $\lambda$-Hurwitz product} $fg$
of any $f, g \in A^{\NN}$ by
\begin{equation}
(fg)(n) = \sum_{k=0}^{n}\sum_{j=0}^{n-k}\binc{n}{k}\binc{n-k}{j}
\lambda^{k}f(n-j)g(k+j).
\mlabel{eq:hurprod}
\end{equation}

We note that the definition of the $\lambda$-Hurwitz product is motivated by Proposition~\ref{prop:Leibnitz}.
In~\cite{GK3} we denoted the algebra $A^{\NN}$ with this product by $D(A)$, but in this paper we will keep the notation $A^{\NN}$, since it indicates the functorial nature of this algebra as functions from $\NN$ to $A$.
As before in~\cite{GK3} we call $A^{\NN}$ the {\bf algebra of $\lambda$-Hurwitz series} over $A$.
Further define a map
$$\partial_A:A^{\NN} \rightarrow A^{\NN},$$
by $\partial_A(f)(n) = f(n+1)$ for any $f \in A^{\NN}$.
Then $\partial_A$ is a $\lambda$-derivation on $A^{\NN}$ and $(A^{\NN},\partial_A)$ is a $\lambda$-differential algebra. Suppose that $h:A \rightarrow B$ is an algebra homomorphism. We define a map $h^{\NN}: A^{\NN} \rightarrow B^{\NN}$ by $(h^{\NN}(f))(n) = h(f(n))$ for any $f\in A^{\NN}$ and $n\in\NN$. It is easy to check that $h^{\NN}$ is a $\lambda$-differential algebra homomorphism from $(A^{\NN},\partial_A)$ to $(B^{\NN},\partial_B)$.

Let $\DIF$ denote the category
of $\lambda$-differential algebras. We see that we have a functor
$G:\ALG \rightarrow \DIF$ given on objects
$A \in \ALG$ by $G(A) = (A^{\NN}, \partial_A)$ and on morphisms
$h:A \rightarrow B$ in $\ALG$ by $G(h) = h^{\NN}$ as defined
above.  Let $V:\DIF \rightarrow \ALG$ denote the forgetful
functor defined on objects $(R,d) \in \DIF$ by $V(R,d) = R$ and
on morphisms $g:(R,d) \rightarrow (S,e)$ in $\DIF$ by $V(g) = g$.

Next, we define two natural
transformations $\eta: \id_{\DIF} \rightarrow G V$ and
$\varepsilon: V G \rightarrow \id_{\ALG}$.
For any $(R,d) \in \DIF$, define
$$\eta_{(R,d)}:(R,d) \rightarrow (G V)(R,d)=(R^{\NN}, \partial_R), \quad (\eta_{(R,d)}(x))(n)\colon= d^{(n)}(x), x\in R, n\in \NN.$$
For any $A\in \ALG$, define
$$\varepsilon_A: (V G)(A) = A^{\NN} \rightarrow A, \quad \varepsilon_A(f)\colon = f(0), f \in A^{\NN}.$$

\begin{prop}$($\cite[Proposition~2.8]{GK3}$)$
The functor $G:\ALG \rightarrow \DIF$ defined above is the right adjoint of the forgetful functor $V:\DIF \rightarrow \ALG$. It follows that $(A^{\NN},\partial_A)$ is a cofree $\lambda$-differential algebra on the algebra $A$.
\mlabel{prop:cofree}
\end{prop}

Proposition~\ref{prop:cofree}
gives an adjunction $\langle V,G,\eta,\varepsilon\rangle:\DIF\rightharpoonup\ALG$.
Corresponding to the adjunction,
there is a comonad $\C = \langle C,\varepsilon,\delta \rangle$ on the category
$\ALG$, where $C$ is the functor
$$C := VG:\ALG \rightarrow \ALG$$
whose value for any $A \in \ALG$ is $C(A) = A^{\NN}$ and $\delta$ is the natural transformation from $C$ to $C C$ defined by $\delta := V\eta G$. In other words, for any $A \in \ALG$,
$$\delta_A:A^{\NN} \rightarrow (A^{\NN})^{\NN}, \quad (\delta_A(f)(m))(n) = f(m+n), \quad f \in A^{\NN}, m, n \in \NN.$$ Note that as a $\bfk$-module, $(A^{\NN})^{\NN}\cong A^{\NN\times \NN}$, the set of sequences of sequences, or equivalently, doubly-indexed sequences with values in $A$.

The comonad $\C$ induces a category
of $\C$-coalgebras, denoted by $\ALG_{\C}$.  The objects in
$\ALG_{\C}$ are pairs $\langle A,f\rangle$ where $A \in \ALG$ and
$f:A \rightarrow A^{\NN}$ is a homomorphism in $\ALG$ satisfying the two properties
$$\varepsilon_A \circ f= \id_A, \quad \delta_A \circ f= f^{\NN} \circ f.$$
A morphism $\varphi:\langle A,f\rangle\rightarrow\langle B,g\rangle$ in $\ALG_{\C}$ is an algebra homomorphism $\varphi: A \rightarrow B$ such that $g \circ \varphi = \varphi^{\NN} \circ f$.

The comonad $\C$ also gives rise to an adjunction
$$\langle V_{\C}, G_{\C}, \eta_{\C}, \varepsilon_{\C} \rangle: \ALG_{\C}
\rightharpoonup \ALG,$$
where
$$V_{\C} : \ALG_{\C} \rightarrow \ALG$$ is given on objects
by $V_{\C}\langle R,f\rangle = R$ and on morphisms
$\varphi:\langle A,f\rangle\rightarrow\langle B,g\rangle$ in $\ALG_{\C}$ by $V_{\C}(\varphi)=\varphi$.  The functor
$$G_{\C}:\ALG \rightarrow \ALG_{\C}$$ is defined on objects
$A \in \ALG$ by $G_{\C}(A) = \langle A^{\NN}, \delta_A\rangle$, and on morphisms
$\phi:A\rightarrow B$ in $\ALG$ by $G_{\C}(\phi)= \phi^{\NN}:A^{\NN}\to B^{\NN}$.  The natural transformations $\varepsilon_{\C}$ and
$\eta_{\C}$ are defined similarly to $\varepsilon$ and $\eta$.

Consequently there is a uniquely defined cocomparison functor
$H:\DIF \rightarrow \ALG_{\C}$ such that $H  G = G_{\C}$ and $V_{\C} H = V$.  Here $H(R,d) = \langle R,\tilde{d}\rangle$,
where for the $\lambda$-derivation
$d:R \rightarrow R$, the algebra homomorphism $\tilde{d}:R \rightarrow R^{\NN}$ is defined by
$\tilde{d} = V(\eta_{(R,d)})$, called the $\lambda$-Hurwitz homomorphism of $d$.
Hence, for any
$a \in R$ and $n \in \NN$, $(\tilde{d}(a))(n) = d^{(n)}(a)$.

\begin{theorem}
The cocomparison functor $H:\DIF \rightarrow \ALG_{\C}$ is
an isomorphism, i.e., $\DIF$ is comonadic over $\ALG$.
\mlabel{thm:comonadic}
\end{theorem}

\begin{proof}
The proof, which uses the dual of Beck's Theorem~\cite{Ma} to show that $H$ is an isomorphism, is virtually the dual of the proof of Theorem~\ref{thm:rbmon} and is omitted.
\end{proof}

\begin{coro}
For any algebra $A$, there is a one-to-one correspondence
between
\begin{enumerate}
\item
$\lambda$-derivations $d$ on $A$ over $\bfk$;
\mlabel{der}
\item
$\C$-costructures $f$ on $A$, i.e., algebra homomorphisms
$f :A \rightarrow A^{\NN}$ satisfying $\varepsilon_A \circ f= \id_A$
and $\delta_A \circ f= f^{\NN} \circ f$;
\mlabel{costr}
\item
sequences of $\bfk$-module homomorphisms $(f_n):A \rightarrow A$ for $n\in \NN$ that
satisfy $f_0 = \id_A$, $f_m \circ f_n = f_{m+n}$ and
$f_n(ab) = \sum\limits_{k=0}^{n}\sum\limits_{j=0}^{n-k}\binc{n}{k}\binc{n-k}{j}
\lambda^{k}f_{n-j}(a)f_{k+j}(b)$ for all $a,b \in A$.
\mlabel{seq}
\end{enumerate}
\mlabel{coro:1-1}
\end{coro}
\begin{proof}
The equivalence of Item~(\ref{der}) and Item~(\ref{costr}) is immediate from Theorem~\ref{thm:comonadic}.

To prove the equivalence of Item~(\ref{costr}) and Item~(\ref{seq}), first note the bijection of linear maps
\begin{eqnarray*}
\{f: A\to A^{\NN}\} &\longleftrightarrow & \{f_n:A\to A, n\geq 0\}, \\
 f &\longmapsto& f_n:=p_n \circ f, \\
 (f(a))(n):=f_n(a)  & \longleftarrow\hspace{-4pt}\raisebox{1.7pt}{\scalebox{0.4}{$\mid\hspace{-3pt}\mid\hspace{-3pt}\mid$}}
  & f_n,
\end{eqnarray*}
where
$$p_n:A^{\NN}\to A, \quad p_n(f):= f(n),$$
is the projection to the $n$-th components, that is, the evaluation at $n$.
One next verifies that under this bijection, an $f$ satisfies $\varepsilon\circ f=\id_A$ and $\delta_A\circ f=f^{\NN} \circ f$ if and only if the corresponding $(f_n)$ satisfies $f_0=\id_A$ and $f_m\circ f_n=f_{m+n}$ respectively. One finally checks that $f$ is an algebra homomorphism if and only if the corresponding $(f_n)$ satisfies $f_n(ab) = \sum\limits_{k=0}^{n}\sum\limits_{j=0}^{n-k}\binc{n}{k}\binc{n-k}{j}
\lambda^{k}f_{n-j}(a)f_{k+j}(b)$ for all $a,b \in A$ by the $\lambda$-Hurwitz product formula in Eq.~(\mref{eq:hurprod}).
\end{proof}

The following result will be used in the remainder of the paper.

\begin{lemma}
For any $f, g\in A^{\NN}$, $f = g$ if and only if $\varepsilon_A (f) = \varepsilon_A (g)$ and $\partial_A (f) = \partial_A (g)$. For any set $C$ and any maps $F, G: C\to A^{\NN}$, $F=G$ if and only if $\varepsilon_A \circ F =\varepsilon_A\circ G$ and $\partial_A\circ F =\partial_A\circ G$.
\mlabel{lem:head-tail}
\end{lemma}

\begin{proof}
For the first statement we only need to prove the ``if" part. But from $\varepsilon_A (f) = \varepsilon_A (g)$ we obtain $f(0)=g(0)$ and from $\partial_A (f) = \partial_A (g)$ we obtain $f(n)=g(n)$ for $n\geq 1$. Then the second statement follows.
\end{proof}

We now extend a Rota-Baxter operator on $A$ to one on $A^{\NN}$.

\begin{prop}
Let $P$ be a Rota-Baxter operator on $A$.
Define a $\bfk$-linear operator $$\widetilde{P}:A^{\NN}\rightarrow A^{\NN}, \quad
\widetilde{P}(f)(0)=P(f(0)), \quad \widetilde{P}(f)(n)=f(n-1),\quad f\in A^{\NN}, n\in \NN_{+}.$$  Then  $\widetilde{P}$ is a Rota-Baxter operator of weight $\lambda$ on $A^{\NN}$, $\varepsilon_A \circ \widetilde{P} = P \circ \varepsilon_A$ and $\partial_A \circ \widetilde{P} = {\rm id}_{A^{\NN}}.$
\mlabel{pp:rbext}
\end{prop}
\begin{proof}
It is clear that $\widetilde{P}$ is a $\bfk$-linear operator on $A^{\NN}$.
For any $f \in A^{\NN}$, we have
$$(\varepsilon_A \circ \widetilde{P})(f) = \varepsilon_A (\widetilde{P}(f)) = \widetilde{P}(f)(0) = P(f(0)) = P(\varepsilon_A(f)) = (P \circ \varepsilon_A)(f),$$
so that $\varepsilon_A \circ \widetilde{P} = P \circ \varepsilon_A$.
Next, for $f \in A^{\NN}$ and $n \in \NN$, we have
$$(\partial_A \circ \widetilde{P})(f)(n) = (\partial_A  (\widetilde{P}(f)))(n) = (\widetilde{P}(f))(n+1) = f(n+1-1) = f(n),$$ so that $\partial_A \circ \widetilde{P} = {\rm id}_{A^{\NN}}.$

We next show that $\widetilde{P}$ is a Rota-Baxter operator of weight $\lambda$ on $A^{\NN}$, i.e., that the equation
$$\widetilde{P}(f) \widetilde{P}(g) = \widetilde{P}(\widetilde{P}(f) g)+\widetilde{P}(f \widetilde{P}(g))+
\lambda\widetilde{P}(f g)$$
holds for any $f, g \in A^{\NN}$. By Lemma~\ref{lem:head-tail}, it's enough to show that
\begin{enumerate}
\item $\partial_A(\widetilde{P}(f) \widetilde{P}(g)) = \partial_A\Big( \widetilde{P}(\widetilde{P}(f) g)+\widetilde{P}(f \widetilde{P}(g))+
\lambda\widetilde{P}(f g)\Big)$, and
\item $\varepsilon_A(\widetilde{P}(f) \widetilde{P}(g)) = \varepsilon_A\Big( \widetilde{P}(\widetilde{P}(f) g)+\widetilde{P}(f \widetilde{P}(g))+
\lambda\widetilde{P}(f g)\Big)$.
\end{enumerate}
The first item follows from $\partial_A \circ \widetilde{P} = {\rm id}_{A^{\NN}}$ and the fact that $\partial_A$ is a $\lambda$-derivation on $A^{\NN}.$
The second follows from the fact that $\varepsilon_A : A^{\NN} \rightarrow A$ is an algebra homomorphism, that $\varepsilon_A(\widetilde{P}(f)) = P(\varepsilon_A(f))$ for any $f \in A^{\NN}$ and that $P$ is a Rota-Baxter operator on $A$.
\end{proof}

\begin{lemma}
If $f : (A, P) \rightarrow (B, Q)$ is a morphism of Rota-Baxter algebras, then $f^{\NN} : (A^{\NN}, \widetilde{P}) \rightarrow (B^{\NN}, \widetilde{Q})$ is also a morphism of Rota-Baxter algebras. Also, $\varepsilon_B \circ f^{\NN} = f \circ \varepsilon_A$ and $\partial_B \circ f^{\NN} = f^{\NN} \circ \partial_A$.
\mlabel{lem:Nmorph}
\end{lemma}

\begin{proof}  We first show that $\varepsilon_B \circ f^{\NN} = f \circ \varepsilon_A$.
Suppose that $h \in A^{\NN}$.
Then
$$(\varepsilon_B \circ f^{\NN})(h) = \varepsilon_{B}(f^{\NN}(h)) = (f^{\NN}(h))(0) = f(h(0)) = f(\varepsilon_A(h)) = (f \circ \varepsilon_A)(h).$$
The proof of $\partial_B \circ f^{\NN} = f^{\NN} \circ \partial_A$ is similar.

To show that $f^{\NN}$ is a morphism of Rota-Baxter algebra, once again we use Lemma~\ref{lem:head-tail}.  We need to show that $\widetilde{Q} \circ f^{\NN} = f^{\NN} \circ \widetilde{P}$.  So we will show both

\begin{enumerate}
\item $\varepsilon_B \circ \widetilde{Q} \circ f^{\NN} = \varepsilon_B \circ f^{\NN} \circ \widetilde{P}$, and
\item $\partial_B \circ \widetilde{Q} \circ f^{\NN} = \partial_B \circ f^{\NN} \circ \widetilde{P}$.
\end{enumerate}
For the first equation,
$$\varepsilon_B \circ \widetilde{Q} \circ f^{\NN} =
Q \circ \varepsilon_B \circ f^{\NN} =
Q \circ f \circ \varepsilon_A  =
f \circ P \circ \varepsilon_A  =
f \circ \varepsilon_A \circ \widetilde{P}  =
\varepsilon_B \circ f^{\NN} \circ \widetilde{P}.$$
For the second equation,
$$\partial_B \circ \widetilde{Q} \circ f^{\NN} = {\rm id}_{B^{\NN}} \circ f^{\NN} =
f^{\NN} \circ {\rm id}_{A^{\NN}}  =
f^{\NN} \circ \partial_A \circ \widetilde{P}  =
\partial_B \circ f^{\NN} \circ \widetilde{P}.$$
\end{proof}

Additional properties of $\lambda$-differential algebras, and of the cofree $\lambda$-differential algebra $(A^{\NN},\partial_A)$ of $\lambda$-Hurwitz series over $A$, will be considered in a subsequent paper.

\section{Mixed distributive laws}
\mlabel{sec:mix}
In this section, we establish a mixed distributive law of the monad $\T$ giving Rota-Baxter algebras over the comonad $\C$ giving differential algebras. This mixed distributive law is described in Section~\mref{ss:mdl}. The prove of the mixed distributive law is given in Section~\mref{ss:mdlproof}.

\subsection{The mixed distributive law for Rota-Baxter algebras over differential algebras}
\mlabel{ss:mdl}
We first recalling from~\cite{HW} some background information on mixed distributive laws, a generalization of the notion of a distributive law introduced by J. Beck in his fundamental work~\cite{Be}.
\begin{defn}
Given a category $\A$, a monad ${\mathbf{T}} = \langle T, \eta, \mu \rangle$ on $\A$ and a comonad $\C = \langle C, \varepsilon, \delta\rangle$ on $\A$, then
a {\bf mixed distributive law of $\T$ over $\C$} is a natural transformation $\beta: TC \rightarrow CT$ such that

\begin{enumerate}
\item $\beta \circ \eta C = C\eta$;
\item $\varepsilon T \circ \beta = T\varepsilon$;
\item $\delta T \circ \beta = C\beta \circ \beta C \circ T\delta$ and
\item $\beta \circ \mu C = C\mu \circ \beta T \circ T\beta$.
\end{enumerate}
That is, the diagrams

\[\xymatrix{
&C\ar[dr]^{C\eta} \ar[dl]_{\eta C}&\\
TC\ar[rr]_{\beta}\ar[dr]_{T\varepsilon}&& CT
\ar[dl]^{\varepsilon T}\\
&T&
}\]
\[
\begin{array}{c}
\xymatrix{
TC\ar[rr]^{\beta} \ar[d]_{T\delta} && CT \ar[d]^{\delta T}\\
TCC\ar[r]_{\beta C}&CTC \ar[r]_{C \beta }& CCT
}
\end{array}
 \text{and}
\begin{array}{c}
\xymatrix{
TTC\ar[r]^{T\beta} \ar[d]_{\mu C} & TCT \ar[r]^{\beta T}&CTT \ar[d]^{C \mu}\\
TC\ar[rr]_{\beta}&& CT
}
\end{array}
\]
commute.
\mlabel{def:mdl}
\end{defn}

Such a mixed distributive law gives rise to a comonad $\widetilde{\C}$ on the category $\A^{\T}$ of $\T$-algebras which lifts $\C$ and a monad $\widetilde{\T}$ on the category $\A_{\C}$ of $\C$-coalgebras which lifts $\T$. See Corollary~\mref{coro:lifting} for the precise meaning of lifting.
Furthermore $(\A_{\C})^{\widetilde{\T}} \cong (\A^{\T})_{\widetilde{\C}}$.
We will apply this to the specific situation where $\A = \ALG$ and $\T$ and $\C$ are the monad giving Rota-Baxter algebras and the comonad giving differential algebras, respectively.

Recall that the adjunction $$\langle V,G,\eta,\varepsilon\rangle:\DIF\rightharpoonup \ALG$$
gives rise to a comonad $\C=\langle C,\varepsilon,\delta\rangle$ on $\ALG$. For any $A\in \ALG$, $C(A)$ is the algebra of Hurwitz series $D(A)$ with the $\lambda$-Hurwitz multiplication, which we continue to write as $A^{\NN}$.

Similarly, the adjunction $$\langle F,U,\eta,\varepsilon\rangle:\ALG\rightharpoonup \RBA$$
gives rise to a monad $\T=\langle T,\eta,\mu\rangle$ on $\ALG$.  Here
$T(A)$ is the algebra $\sha(A)$ with the multiplication $\diamond$.

By the definitions of $C$ and $T$, $C(T(A))$ is the algebra $(\sha(A))^{\NN}$.
By Proposition~\mref{pp:rbext}, the Rota-Baxter operator $P_A$ on $\sha(A)$ induces a Rota-Baxter operator $\widetilde{P_A}$ on $(\sha(A))^{\NN}$,
so that $((\sha(A))^{\NN}, \widetilde{P_A})$ is a Rota-Baxter algebra.

\begin{lemma}
For any algebra $A$, there is a unique Rota-Baxter algebra homomorphism $$\beta_{A}:(\sha(A^{\NN}), P_{A^{\NN}})\rightarrow ((\sha(A))^{\NN}, \widetilde{P_{A}})$$
such that the equation
\begin{equation}
(\eta_{A})^{\NN}=\beta_{A}\circ \eta_{A^{\NN}}
\mlabel{eq:c0}
\end{equation}
holds.
\mlabel{lem:beta}
\end{lemma}

\begin{proof}
Recall that $\eta_A: A \rightarrow \sha(A)$ is a natural algebra homomorphism.
Moreover, $(\eta_{A})^{\NN}:A^{\NN}\rightarrow (\sha(A))^{\NN}$ is an algebra homomorphism.
By the universal property of the free
Rota-Baxter algebra on $A^{\NN}$, there is a unique Rota-Baxter algebra homomorphism
$$\beta_{A}:(\sha(A^{\NN}), P_{A^{\NN}})\rightarrow ((\sha(A))^{\NN}, \widetilde{P_{A}})$$
such that Eq.~(\mref{eq:c0}) holds.
\end{proof}

\begin{lemma}
The above defined $\beta_{A}:\sha(A^{\NN})\rightarrow (\sha(A))^{\NN}$ for any $A\in \ALG$ gives a natural transformation $\beta:TC\rightarrow CT$.
\mlabel{lem:nt}
\end{lemma}

\begin{proof}
It is enough to show that for any morphism $\varphi:A\rightarrow B$ in $\ALG$, the equation
\begin{equation}
(\sha(\varphi))^{\NN}\circ \beta_A=\beta_{B}\circ \sha(\varphi^{\NN})
\mlabel{eq:c1}
\end{equation}
holds.

By the naturality of $\eta$, the equations
\begin{equation}\sha(\varphi^{\NN})\circ \eta_{A^{\NN}}=\eta_{B^{\NN}}\circ \varphi^{\NN}
\mlabel{eq:c2}
\end{equation}
and
\begin{equation}\sha(\varphi)\circ \eta_{A}=\eta_{B}\circ \varphi
\mlabel{eq:c3}
\end{equation}
hold.

We claim that $(\sha(\varphi))^{\NN}:(\sha(A))^{\NN} \rightarrow (\sha(B))^{\NN}$ is a Rota-Baxter algebra homomorphism.  This is a consequence of Lemma~\ref{lem:Nmorph}, since $\sha(\varphi):\sha(A) \rightarrow \sha(B)$ is a Rota-Baxter algebra homomorphism.

Because $\beta_A$, $\beta_B$ and $(\sha(\varphi))^{\NN}$ are Rota-Baxter algebra homomorphisms,
the compositions $(\sha(\varphi))^{\NN}\circ \beta_A$ and $\beta_B\circ \sha(\varphi^{\NN})$
are Rota-Baxter algebra homomorphisms, too.

Next, we have
\begin{eqnarray*}
((\sha(\varphi))^{\NN}\circ \beta_A)\circ \eta_{A^{\NN}}&=&(\sha(\varphi))^{\NN}\circ (\eta_{A})^{\NN} \quad \text{(by Eq.~(\ref{eq:c0}))}   \\
&=&(\sha(\varphi) \circ \eta_{A})^{\NN}\\
&=&(\eta_{B} \circ \varphi)^{\NN}\quad \quad \quad \quad \text{(by Eq.~(\ref{eq:c3}))} \\
&=&(\eta_{B})^{\NN}\circ \varphi^{\NN}  \\
&=&(\beta_B\circ \eta_{B^{\NN}})\circ \varphi^{\NN}\quad \quad \text{(by Eq.~(\ref{eq:c0}))}   \\
&=&\beta_B\circ (\eta_{B^{\NN}}\circ \varphi^{\NN})\\
&=&\beta_{B}\circ(\sha(\varphi^{\NN})\circ \eta_{A^{\NN}})\quad \text{(by Eq.~(\ref{eq:c2}))}   \\
&=&(\beta_{B}\circ \sha(\varphi^{\NN}))\circ \eta_{A^{\NN}}.
\end{eqnarray*}
By the universal property of a free
Rota-Baxter algebra on the algebra $A^{\NN}$,
we have
$$(\sha(\varphi))^{\NN}\circ \beta_A = \beta_{B}\circ \sha(\varphi^{\NN}),$$
as desired.
\end{proof}

Now we can state our main result of this section.
\begin{theorem}
The natural transformation $\beta: TC \rightarrow CT$ given by  $\beta_{A}:\sha(A^{\NN})\rightarrow (\sha(A))^{\NN}$
is a mixed distributive law of\, $\T$ over $\C$.
\mlabel{thm:cddist}
\end{theorem}

We will postpone the proof of the theorem to the next subsection and  first display some direct applications of the theorem. Further applications of the theorem will be given in Section~\mref{sec:diffrba}.
In particular we will investigate the categories $(\ALG_{\C})^{\widetilde{\T}} \cong (\ALG^{\T})_{\widetilde{\C}}$ and identify them in terms of a ``mixed" structure on $\ALG$.

\begin{coro}
The mixed distributive law $\beta: TC \rightarrow CT$ in Theorem~\ref{thm:cddist}
gives rise to a comonad $\widetilde{\C}$ on the category $\ALG^{\T}$ of $\T$-algebras which {\bf lifts $\C$} in the sense that the underlying functor $U^{\T}: \ALG^{\T} \rightarrow \ALG$ commutes with $\widetilde{\C}$ and $\C$, that is,
$$U^{\T} \widetilde{C} = C U^{\T}, \quad U^{\T}\widetilde{\varepsilon} = \varepsilon U^{\T}\quad \text{ and }\quad U^{\T}\widetilde{\delta} = \delta U^{\T}.$$
Similarly, the mixed distributive law $\beta$ gives rise to a monad $\widetilde{\T}$ on the category $\ALG_{\C}$ of $\C$-coalgebras which lifts $\T$.
Furthermore there is an isomorphism $\Phi:(\ALG_{\C})^{\widetilde{\T}} \to (\ALG^{\T})_{\widetilde{\C}}$ of categories.
\mlabel{coro:lifting}
\end{coro}

\begin{proof}
This result is an immediate consequence of~\cite[Lemma 3.1]{Ke3}, ~\cite[Theorem IV.1]{Va}, ~\cite[Theorem 2.4]{Wolff} and Theorem~\ref{thm:cddist} above.
\end{proof}

\subsection{The proof of Theorem~\mref{thm:cddist}}
\mlabel{ss:mdlproof}
To prove the theorem, we just need to verify the commutativity of the three diagrams in the definition of a mixed distributive law. These verifications are carried out in the following three lemmas.

\begin{lemma}
The diagram
\begin{equation}
\begin{array}{c}
\xymatrix{
&A^{\NN}\ar[dr]^{(\eta_{A})^{\NN}} \ar[dl]_{\eta_{A^{\NN}}}&\\
\sha(A^{\NN})\ar[rr]_{\beta_{A}}\ar[dr]_{\ssha(\varepsilon_{A})}&& (\sha(A))^{\NN}\ar[dl]^{\varepsilon_{\tsha(A)}}\\
&\sha(A)&
}
\end{array}
\mlabel{eq:23j}
\end{equation}
commutes.
\mlabel{lem:23jiao}
\end{lemma}
\begin{proof}
By Lemma~\ref{lem:beta}, the upper triangle in diagram~(\ref{eq:23j}) is commutative.

Next observe that $\varepsilon_{\tsha{(A)}}: (\sha(A))^{\NN} \rightarrow \sha(A)$ is a Rota-Baxter algebra homomorphism by Proposition~\ref{pp:rbext}.
We know that $\beta_A$ is a Rota-Baxter algebra homomorphism and so
the composition $\varepsilon_{\tsha(A)}\circ \beta_{A}$ is also a Rota-Baxter algebra homomorphism.

By the universal property of the free Rota-Baxter algebra $\sha(A^{\NN})$ over $A^{\NN}$, to prove the commutativity of the lower triangle:
\begin{equation}
\sha(\varepsilon_{A} )=\varepsilon_{\tsha(A)}\circ \beta_{A},
\mlabel{eq:j6}
\end{equation}
we only need to verify
\begin{equation}
\sha(\varepsilon_{A} )\circ \eta_{A^{\NN}}=(\varepsilon_{\tsha(A)}\circ \beta_{A})\circ \eta_{A^{\NN}}.
\mlabel{eq:j5}
\end{equation}
But by the naturality of $\varepsilon$, we have
\begin{equation}
\varepsilon_{\tsha(A)}\circ (\eta_{A})^{\NN}= \eta_{A}\circ \varepsilon_{A}.
\mlabel{eq:j1}
\end{equation}
By the naturality of $\eta$, we have
\begin{equation}
\sha(\varepsilon_{A})\circ \eta_{A^{\NN}}=\eta_{A}\circ \varepsilon_{A}.
\mlabel{eq:j2}
\end{equation}
By Eq.~(\ref{eq:j1}) and Eq.~(\ref{eq:j2}), we have
\begin{equation}
\sha(\varepsilon_{A})\circ \eta_{A^{\NN}}=\varepsilon_{\tsha(A)}\circ (\eta_{A})^{\NN}.
\mlabel{eq:j3}
\end{equation}
By Eq.~(\ref{eq:j6}), we have
\begin{equation}
(\varepsilon_{\tsha(A)}\circ \beta_{A})\circ \eta_{{A}^{\NN}}=
\varepsilon_{\tsha(A)}\circ (\eta_{A})^{\NN}.
\mlabel{eq:j4}
\end{equation}
Then by Eq.~(\ref{eq:j3}) and Eq.~(\ref{eq:j4}), we obtain Eq.~(\mref{eq:j5}), as needed.
\end{proof}

\begin{lemma}
The diagram
\begin{equation}
\begin{array}{c}
\xymatrix{
\sha(A^{\NN})\ar[rr]^{\beta_{A}} \ar[d]_{\ssha(\delta_{A})} && (\sha(A))^{\NN} \ar[d]^{\delta_{\tsha(A)}}\\
\sha((A^{\NN})^{\NN})\ar[r]_{\beta_{A^{\NN}}}&
(\sha(A^{\NN}))^{\NN} \ar[r]_{(\beta_{A})^{\NN} }& ((\sha(A))^{\NN})^{\NN}
}
\end{array}
\mlabel{eq:1f}
\end{equation}
commutes.

\mlabel{lem:1fang}
\end{lemma}

\begin{proof}
We know that $\beta_{A}$ is a Rota-Baxter algebra homomorphism, i.e.,
\begin{equation}
\beta_{A}\circ P_{A^{\NN}}= \widetilde{P_{A}}\circ\beta_{A}.
\mlabel{eq:1f1}
\end{equation}
Then from Lemma~\ref{lem:Nmorph} we see that $(\beta_{A})^{\NN}$ is also a Rota-Baxter algebra homomorphism.

We next verify that $\delta_{\tsha(A)} : (\sha(A))^{\NN} \rightarrow ((\sha(A))^{\NN})^{\NN}$ is a Rota-Baxter algebra homomorphism.
This means we need to verify that $$\widetilde{(\widetilde{P_A})} \circ \delta_{\ssha(A)} = \delta_{\ssha(A)} \circ \widetilde{P_A}.$$
Once again we make use of Lemma~\ref{lem:head-tail}.  So we observe that
$$\partial_{(\ssha(A))^{\NN}} \circ \widetilde{(\widetilde{P_A})} \circ \delta_{\ssha(A)} =
{\rm id}_{((\ssha(A))^{\NN})^{\NN}} \circ \delta_{\ssha(A)} =
\delta_{\ssha(A)} \circ {\rm id}_{(\ssha(A))^{\NN}} =
\delta_{\ssha(A)} \circ \partial_{\ssha(A)} \circ \widetilde{P_A} =
 \partial_{(\ssha(A))^{\NN}} \circ \delta_{\ssha(A)} \circ \widetilde{P_A}.$$

In a similar way we can also show
$$\varepsilon_{(\ssha(A))^{\NN}} \circ \widetilde{(\widetilde{P_A})} \circ \delta_{\ssha(A)} = \varepsilon_{(\ssha(A))^{\NN}} \circ \delta_{\ssha(A)} \circ \widetilde{P_A}.$$
To see this, first note that $\varepsilon_{(\tsha(A))^{\NN}} \circ \widetilde{\widetilde{P_A}} = \widetilde{P_A} \circ \varepsilon_{(\tsha(A))^{\NN}}$ by Proposition~\ref{pp:rbext}.  Then $$\varepsilon_{(\ssha(A))^{\NN}} \circ \widetilde{(\widetilde{P_A})} \circ \delta_{\ssha(A)} =
\widetilde{P_A} \circ \varepsilon_{(\tsha(A))^{\NN}} \circ \delta_{\ssha(A)} =
\widetilde{P_A} \circ {\rm id}_{(\tsha(A))^{\NN}}  =
{\rm id}_{(\tsha(A))^{\NN}} \circ \widetilde{P_A}  =
\varepsilon_{(\ssha(A))^{\NN}} \circ \delta_{\ssha(A)} \circ \widetilde{P_A}.$$
Therefore, $\delta_{\tsha(A)} : (\sha(A))^{\NN} \rightarrow ((\sha(A))^{\NN})^{\NN}$ is a Rota-Baxter algebra homomorphism.
Consequently we see that $\delta_{\tsha(A)}\circ\beta_{A}$ and $ (\beta_{A})^{\NN}\circ\beta_{A_{\NN}}\circ \sha(\delta_{A} )$ are Rota-Baxter algebra homomorphisms.

By the universal property of a free
Rota-Baxter algebra $\sha(A^{\NN})$ on $A^{\NN}$,
to prove the commutativity of the diagram (\mref{eq:1f}), namely,
$$\delta_{\tsha(A)}\circ\beta_{A}= (\beta_{A})^{\NN}\circ\beta_{A^{\NN}}\circ \sha(\delta_{A}),$$
we only need to prove
$$(\delta_{\tsha(A)}\circ\beta_{A})\circ \eta_{A^{\NN}}= ((\beta_{A})^{\NN}\circ\beta_{A^{\NN}}\circ \sha(\delta_{A}))\circ \eta_{A^{\NN}}.$$

Now by the naturality of $\delta$, we have
\begin{equation}
\delta_{\tsha(A)}\circ (\eta_{A})^{\NN}=((\eta_{A})^{\NN})^{\NN}\circ \delta_{A}.
\mlabel{eq:1f4}
\end{equation}
By the naturality of $\eta$, we have
\begin{equation}
\eta_{(A^{\NN})^{\NN}}\circ \delta_{A}=\sha(\delta_{A})\circ \eta_{A^{\NN}}.
\mlabel{eq:1f5}
\end{equation}
Then
\begin{eqnarray*}
(\delta_{\tsha(A)}\circ\beta_{A})\circ \eta_{A^{\NN}}&=&\delta_{\tsha(A)}\circ(\beta_{A}\circ \eta_{A^{\NN}}) \\
&=&\delta_{\tsha(A)}\circ (\eta_{A})^{\NN} \quad \quad \text{(by Eq.~(\ref{eq:c0}))}\\
&=&((\eta_{A})^{\NN})^{\NN}\circ \delta_{A} \quad \quad \text{(by Eq.~(\ref{eq:1f4}))}\\
&=&(\beta_{A}\circ \eta_{A^{\NN}})^{\NN}\circ
\delta_{A} \quad \quad \text{(by Eq.~(\ref{eq:c0}))}\\
&=& (\beta_{A})^{\NN}\circ (\eta_{A^{\NN}})^{\NN}\circ
\delta_{A} \\
&=& (\beta_{A})^{\NN}\circ(\beta_{A^{\NN}}\circ \eta_{(A^{\NN})^{\NN}})\circ
\delta_{A} \quad \quad \text{(by Eq.~(\ref{eq:c0}))}\\
&=& (\beta_{A})^{\NN}\circ\beta_{A^{\NN}}\circ (\eta_{(A^{\NN})^{\NN}}\circ
\delta_{A}) \\
&=& (\beta_{A})^{\NN}\circ\beta_{A^{\NN}}\circ (\sha(\delta_{A})\circ
\eta_{A^{\NN}}) \quad \quad \text{(by Eq.~(\ref{eq:1f5}))}\\
&=& ((\beta_{A})^{\NN}\circ\beta_{A^{\NN}}\circ \sha(\delta_{A}))\circ \eta_{A^{\NN}}.
\end{eqnarray*}
This is what we want.
\end{proof}

\begin{lemma}
The diagram
\begin{equation}
\begin{array}{c}
\xymatrix{
\sha(\sha(A^{\NN}))\ar[r]^{\ssha(\beta_{A})} \ar[d]_{\mu_{A^{\NN}}}
&\sha((\sha(A))^{\NN}) \ar[r]^{\beta_{\tsha(A)}}&(\sha(\sha(A)))^{\NN} \ar[d]^{(\mu_{A})^{\NN}}\\
\sha(A^{\NN})\ar[rr]_{\beta_{A}}&& (\sha(A))^{\NN}
}
\end{array}
\mlabel{eq:dist3}
\end{equation}
commutes.
\mlabel{lem:dist3}
\mlabel{lem:2fang}
\end{lemma}

\begin{proof}
By its construction, $\mu_{A}$ is the Rota-Baxter algebra homomorphism induced from the identity homomorphism $\id_{\ssha(A)}$ by the universal property of the free Rota-Baxter algebra $\sha(\sha(A))$ on $\sha(A)$.
Then $(\mu_{A})^{\NN}$ is a Rota-Baxter algebra homomorphism.
Then $\beta_{A}\circ\mu_{A^{\NN}}$ and $(\mu_{A})^{\NN}\circ\beta_{\tsha(A)}\circ \sha(\beta_{A})$ are
Rota-Baxter algebra homomorphisms.

By the uniqueness of the universal property of a free
Rota-Baxter algebra on $\sha(A^{\NN})$, in order to prove the commutativity of the diagram~(\mref{eq:dist3}), we only need to prove  \begin{equation}
( \mu_{A})^{\NN}\circ\beta_{\tsha(A)}\circ \sha(\beta_{A})=\beta_{A}\circ \mu_{A^{\NN}}
\mlabel{eq:dist3'}
\end{equation}

By the naturality of $\eta$, we have
\begin{equation}
\sha(\beta_{A})\circ \eta_{\tsha(A^{\NN})}=\eta_{(\tsha(A))^{\NN}}\circ\beta_{A}.
\mlabel{eq:2f4}
\end{equation}
By Eq.~(\ref{eq:Tcate}), we have
\begin{equation}
\mu_{A}\circ \eta_{\tsha(A)}=\id_{\tsha(A)}.
\mlabel{eq:2f5}
\end{equation}
Then Eq.~(\mref{eq:dist3'}) follows since
\begin{eqnarray*}
( (\mu_{A})^{\NN}\circ\beta_{\tsha(A)}\circ \sha(\beta_{A}))\circ \eta_{\tsha(A^{\NN})}
&=&(\mu_{A})^{\NN}\circ\beta_{\tsha(A)}\circ ( \sha(\beta_{A})\circ \eta_{\tsha(A^{\NN})})\\
&=&(\mu_{A})^{\NN}\circ\beta_{\tsha(A)}\circ
(\eta_{(\tsha(A))^{\NN}}\circ\beta_{A}) \quad \quad \text{(by Eq.~(\ref{eq:2f4})}\\
&=&(\mu_{A})^{\NN}\circ(\beta_{\tsha(A)}\circ
\eta_{(\tsha(A))^{\NN}})\circ\beta_{A} \\
&=& ( \mu_{A})^{\NN}\circ  (\eta_{\tsha(A)})^{\NN}\circ\beta_{A}
\quad \quad \text{(by Eq.~(\ref{eq:c0}))}\\
&=& ( \mu_{A} \circ  \eta_{\tsha(A)})^{\NN}\circ\beta_{A}\\
&=&(\id_{\tsha(A)})^{\NN}\circ\beta_{A} \quad \quad \text{(by Eq.~(\ref{eq:2f5}))}\\
&=&\beta_{A}\circ \id_{\tsha(A^{\NN})}\\
&=&\beta_{A}\circ (\mu_{A^{\NN}}\circ \eta_{\tsha(A^{\NN})})
\quad \quad \text{(by Eq.~(\ref{eq:2f5}))}\\
&=&(\beta_{A}\circ\mu_{A^{\NN}})\circ \eta_{\tsha(A^{\NN})}.
\end{eqnarray*}
\end{proof}

This completes the proof of Theorem~\mref{thm:cddist}.

\section{Differential Rota-Baxter algebras}
\mlabel{sec:diffrba}
In this section we give some applications of the mixed distributive law to differential Rota-Baxter algebras. We first consider the comonad on Rota-Baxter algebras giving differential Rota-Baxter algebras. We then consider the monad on differential algebra giving differential Rota-Baxter algebras. The main results are summarized in Diagram~(\mref{eq:diag}).

\begin{defn}
We say that $(R,d,P)$ is a
{\bf differential Rota-Baxter algebra of weight $\lambda$} if
\begin{enumerate}
\item $(R,d)$ is a differential algebra of weight $\lambda$,
\item $(R,P)$ is a Rota-Baxter algebra of weight $\lambda$, and
\item $d \circ P = \id_R$.
\end{enumerate}
\end{defn}

If $(R,d,P)$ and $(R',d',P')$ are differential Rota-Baxter algebras of weight $\lambda$, then a morphism of differential Rota-Baxter algebras $f:(R,d,P) \rightarrow (R',d',P')$ is an algebra homomorphism $f: R \rightarrow R'$ such that $d'(f(x)) = f(d(x))$ and $P'(f(x)) = f(P(x))$ for all $x \in R$.  The category of differential Rota-Baxter algebras of weight $\lambda$ will be denoted by $\DRB_{\bfk,\lambda}$ or simply by $\DRB$.

It is clear that there are forgetful functors $U': \DRB \rightarrow \DIF$ and $V': \DRB \rightarrow \RBA$, and that $U  V' = V  U'$.
We will show that $\DRB \cong (\ALG_{\C})^{\widetilde{\T}} \cong (\ALG^{\T})_{\widetilde{\C}}$, where $\widetilde{\T}$ and $\widetilde{\C}$ come from Corollary~\ref{coro:lifting}.
This will ``automatically'' give the lifting of the adjoints.

The comonad $\widetilde{\C}$ on the category $\ALG^{\T}$ of $\T$-algebras is $\widetilde{\C} = \langle \widetilde{C},\widetilde{\varepsilon}, \widetilde{\delta} \rangle$, where $\widetilde{C}: \ALG^{\T} \rightarrow \ALG^{\T}$ is a functor and $\widetilde{\varepsilon}: \widetilde{C} \rightarrow \id_{\ALG^{\T}}$ and $\widetilde{\delta}: \widetilde{C} \rightarrow \widetilde{C} \circ \widetilde{C}$ are natural transformations satisfying the usual comonad equations.
Here $\widetilde{C}: \ALG^{\T} \rightarrow \ALG^{\T}$ is given by $$\widetilde{C}\langle A,h \rangle = \langle C(A), C(h) \circ \beta_A \rangle = \langle A^{\NN}, h^{\NN} \circ \beta_A \rangle $$ for any $\langle A,h \rangle \in \ALG^{\T}$.
The equations needed to show that $h^{\NN} \circ \beta_A $ is a $\T$-structure on $A^{\NN}$ follow from the defining equations for $\beta$ being a mixed distributive law.
In a similar way, the natural transformations $\widetilde{\varepsilon}$ and $\widetilde{\delta}$ are defined by $\widetilde{\varepsilon}_{\langle A,h \rangle} = \varepsilon_A$ and $\widetilde{\delta}_{\langle A,h \rangle} = \delta_A$, and the equations to show that these are morphisms in $\ALG^{\T}$ follow from the defining equations for $\beta$.

On the other hand, the monad $\widetilde{\T} = \langle \widetilde{T}, \widetilde{\eta}, \widetilde{\mu} \rangle $ on $\ALG_{\C}$ has the functor $\widetilde{T} : \ALG_{\C} \rightarrow \ALG_{\C}$ given by $$\widetilde{T} \langle A, f\rangle = \langle T(A), \beta_A \circ T(f) \rangle = \langle \sha(A), \beta_A \circ \sha(f) \rangle $$ for any $\langle A, f\rangle \in \ALG_{\C}$.
Once again, the equations needed to show that $\beta_A \circ \sha(f)$ is a $\C$-costructure on $\sha(A)$ come from the defining equations for $\beta$.
The natural transformations $\widetilde{\eta}$ and $\widetilde{\mu}$ are defined in the obvious way as for $\widetilde{\varepsilon}$ and $\widetilde{\delta}$, and the defining equations for $\beta$ come into play here, too.

Now from Corollary~\ref{coro:lifting} we have the isomorphism of categories $(\ALG_{\C})^{\widetilde{\T}} \cong (\ALG^{\T})_{\widetilde{\C}}$.
First we will look at the objects in each of the categories $(\ALG_{\C})^{\widetilde{\T}}$ and $(\ALG^{\T})_{\widetilde{\C}}$, and we will then see exactly what is the isomorphism.
Then we will examine how $\DRB$ is isomorphic to both categories, and the lifting of the adjoint functors on $\ALG$ giving $\T$ and $\C$ will be clear as well.

The objects in $(\ALG^{\T})_{\widetilde{\C}}$ are pairs $\langle \langle A, h \rangle , f \rangle$, where $\langle A, h \rangle \in \ALG^{\T}$ and $f : \langle A, h \rangle \rightarrow \widetilde{C}\langle A, h \rangle = \langle A^{\NN}, h^{\NN} \circ \beta_A \rangle $ is a $\widetilde{\C}$-costructure on $\langle A, h \rangle$.
What this means is that $f : A \rightarrow A^{\NN}$ is an algebra homomorphism such that $\varepsilon_A \circ f = {\rm id_A}$ and $\delta_A\circ f = f^{\NN}\circ f$ and the following diagram commutes:

\begin{equation}
\begin{array}{c}
\xymatrix{
\sha(A)\ar[rr]^{h} \ar[d]_{\ssha(f)} && A \ar[d]^{f}\\
\sha(A^{\NN})\ar[r]_{\beta_{A}}&
\sha(A^{\NN}) \ar[r]_{h^{\NN} }& A^{\NN}
}
\end{array}
\mlabel{eq:(ALG^T)_C}
\end{equation}

The morphisms in $(\ALG^{\T})_{\widetilde{\C}}$ are $g : \langle \langle A, h \rangle , f \rangle \rightarrow \langle \langle A', h' \rangle , f' \rangle $ where $g : A \rightarrow  A' $ is an algebra homomorphism such that the following diagrams commute:

\[
\begin{array}{c}
\xymatrix{
\sha(A)\ar[r]^(.5){h} \ar[d]_{\ssha(g)} & A \ar[d]^{g}\\
\sha(A')\ar[r]_(.5){h'}& A'
}
\end{array}
\quad
\text{ and } \quad
\begin{array}{c}
\xymatrix{
A \ar[r]^(.5){f} \ar[d]_{g} & A^{\NN} \ar[d]^{g^{\NN}}\\
A' \ar[r]_(.5){f'}&(A')^{\NN}
}
\end{array}
\]

The objects in the category $(\ALG_{\C})^{\widetilde{\T}}$ are of the form $\langle \langle A, f \rangle , h \rangle$, where $\langle A, f \rangle \in \ALG_{\C}$ and $h :\widetilde{T} \langle A, f \rangle  = \langle \sha(A), \beta_A \circ \sha(f) \rangle \rightarrow \langle A, f \rangle $ is a $\widetilde{\T}$-structure on $\langle A, f \rangle$.
In this case this means that $h : \sha(A) \rightarrow A$ is an algebra homomorphism satisfying $h \circ \eta_A = {\rm id_A}$ and $h \circ \mu_A = h \circ \sha(h)$ and such that the following diagram commutes:

\begin{equation}
\begin{array}{c}
\xymatrix{
\sha(A)\ar[r]^{\ssha(f)} \ar[d]_{h}
&\sha(A^{\NN}) \ar[r]^{\beta_A}&(\sha(A))^{\NN} \ar[d]^{h^{\NN}}\\
A \ar[rr]_{f}&& A^{\NN}
}
\end{array}
\mlabel{eq:(ALG_C)^T}
\end{equation}

The morphisms in $(\ALG_{\C})^{\widetilde{\T}}$ are defined in a similar way.
But it is clear that diagram~(\ref{eq:(ALG^T)_C}) is identical to diagram~(\ref{eq:(ALG_C)^T}), and hence it is clear how $(\ALG^{\T})_{\widetilde{\C}}$ is isomorphic to $(\ALG_{\C})^{\widetilde{\T}}$.
The isomorphism $\Phi :   (\ALG^{\T})_{\widetilde{\C}} \rightarrow (\ALG_{\C})^{\widetilde{\T}}$ in Corollary~\mref{coro:lifting} is given on objects $\langle \langle A, h \rangle , f \rangle \in (\ALG^{\T})_{\widetilde{\C}}$ by $\Phi \langle \langle A, h \rangle , f \rangle = \langle \langle A, f \rangle , h \rangle $.

Next we'll show how $\DRB \cong (\ALG_{\C})^{\widetilde{\T}} \cong (\ALG^{\T})_{\widetilde{\C}}$ and this will give the lifting of the adjoints.

Let $A^{\NN}, \partial_A$ and $\widetilde{P}$ be as in Proposition~\mref{pp:rbext}. Since $\partial_A\circ \widetilde{P}=\id_{A^{\NN}}$ as stated there, the triple $(A^{\NN},\partial_A,\widetilde{P})$ is a differential Rota-Baxter algebra. Thus we have a functor
$G':\RBA \rightarrow \DRB$ given on objects
$(A,P) \in \RBA$ by $G'(A,P) = (A^{\NN}, \partial_A,\widetilde{P})$ and on morphisms
$\varphi:(A, P) \rightarrow (A', P')$ in $\RBA$ by $(G'(\varphi)(f))(n) = \varphi(f(n))$ for
$f\in A^{\NN}$ and $n\in \NN$.  Recall that  $V':\DRB \rightarrow \RBA$ denotes the forgetful
functor defined on objects $(R,d,P) \in \DRB$ by $V'(R,d,P) = (R,P)$ and
on morphisms $\phi:(R,d,P) \rightarrow (R',d',P')$ in $\DRB$
by $V'(\phi) = \phi$.

\begin{prop}
The functor $G':\RBA \rightarrow \DRB$ defined above is the right
adjoint of the forgetful functor $V':\DRB \rightarrow \RBA$.
\mlabel{pp:adjDRB}
\end{prop}

\begin{proof}
By~\cite{Ma}, it is equivalent to show that there are two natural
transformations $\eta': \id_{\DRB} \rightarrow G' \circ V'$ and
$\varepsilon': V' \circ G' \rightarrow \id_{\RBA}$ satisfying the equations
$$G'\varepsilon' \circ \eta' G' = G', \quad \varepsilon' V' \circ V'\eta' = V'.$$

For any $(A,P) \in \RBA$, define
$$\varepsilon'_{(A,P)}: V'G'(A,P) = (A^{\NN},\widetilde{P}) \rightarrow A, \quad \varepsilon'_{(A,P)}(f) = f(0) = \varepsilon_A(f) \quad \text{ for all }f \in A^{\NN}.$$  Because $\varepsilon_A(\widetilde{P}(f)) = P(\varepsilon_A(f))$
we can easily get that
$\varepsilon'_{(A,P)}$ is a morphism in $\RBA$.
Further if
$\phi:(A,P) \rightarrow (A',P')$ is any morphism of Rota-Baxter algebras, then
$$\varepsilon'_{(A',P')} \circ V'G'(\phi) = \phi \circ \varepsilon'_{(A,P)},$$ i.e., $\varepsilon'$
is a natural transformation as desired.

For any $(R,d,P) \in \DRB$, define
$$\eta'_{(R,d,P)}:(R,d,P) \rightarrow (R^{\NN}, \partial_R,\widetilde{P}), \quad (\eta'_{(R,d,P)}(x))(n) = d^{(n)}(x) = (\eta_{(R,d)}(x))(n) \quad \text{for all } x \in R, n \in \NN.$$
It is not difficult to see
that $\eta'_{(R,d,P)}$ is $\bfk$-linear, and it is immediate from
Proposition~\ref{prop:Leibnitz} that for any $x, y \in R$,
$$(\eta'_{(R,d,P)}(x))(\eta'_{(R,d,P)}(y)) = \eta'_{(R,d,P)}(xy).$$
Also, it is clear that
$$\partial_R \circ \eta'_{(R,d,P)} =
\eta'_{(R,d,P)} \circ d, \quad \widetilde{P}\circ\eta'_{(R,d,P)}=\eta'_{(R,d,P)}\circ P,$$
so that $\eta'_{(R,d,P)}$ is a morphism in $\DRB$.
Furthermore, if $\varphi:(R, d, P) \rightarrow (R',d',P')$ is a morphism in $\DRB$,
then one sees that $\eta'_{(R',d',P')} \circ \varphi = (G'V')\varphi \circ \eta'_{(R,d,P)}.$
Hence $\eta'$ is a natural transformation.

To see that $G'\varepsilon' \circ \eta' G' = G'$, let $(A,P)\in \RBA$,
$f \in R^{\NN}$ and $n \in \NN$.  Then
\begin{eqnarray*}
(G'\varepsilon'_{(R,B)}(\eta'_{(R^{\NN},\partial_R,\widetilde{P})}(f)))(n) &=&
\varepsilon'_{(R,B)}(\eta'_{(R^{\NN},\partial_R,\widetilde{P})}(f)(n)) =
\varepsilon'_{(R,B)}(\partial_{R}^{(n)}(f)) \\
&=&(\partial_{R}^{(n)}(f))(0) =
f(0+n) = f(n).
\end{eqnarray*}

Similarly, to see that $\varepsilon' V' \circ V'\eta' = V'$,
let $(R,d,P) \in \DRB$, and $x \in R$.  Then
$$\varepsilon'_{(R,P)}(\eta'_{(R,d,P)}(x)) =
(\eta'_{(R,d,P)}(x))(0) = d^{(0)}(x) = x.$$
\end{proof}

Proposition~\ref{pp:adjDRB}
gives an adjunction
$$\langle V',G',\eta',\varepsilon'\rangle:\DRB\rightharpoonup\RBA.$$
Corresponding to the adjunction,
there is a comonad $\C' = \langle C',\varepsilon',\delta' \rangle$ on the category
$\RBA$, where
$$C' := V' G':\RBA \rightarrow \RBA$$
is the functor
whose value for any $(A,P) \in \RBA$ is $C'(A,P) = (A^{\NN}, \widetilde{P})$ and $\delta'$ is a natural transformation from $C'$ to $C' C'$ defined by $\delta' := V'\eta' G'$. In other words, for any $(A,P) \in \RBA$, $$\delta'_{(A,P)}:(A^\NN,\widetilde{P}) \rightarrow ((A^{\NN})^{\NN}, \widetilde{\widetilde{P}})\quad \delta'_{(A,P)}(f) =  \delta_A(f), \quad f \in A^{\NN}.$$
Consequently, there is a category
of $\C'$-coalgebras, denoted by $\RBA_{\C'}$.  The objects in
$\RBA_{\C'}$ are pairs $\langle (A,P),f\rangle$ where $(A,P) \in \RBA$ and
$f:(A,P) \rightarrow (A^{\NN}, \widetilde{P})$ is a morphism in $\RBA$ satisfying the two properties
$$\varepsilon'_{(A,P)} \circ f= \id_{(A,P)}, \quad \delta'_{(A,P)} \circ f= f^{\NN} \circ f.$$
A morphism $\varphi:\langle (A,P),f\rangle\rightarrow\langle (B,Q),g\rangle$ in $\RBA_{\C'}$ is an algebra homomorphism $\varphi: A \rightarrow B$ such that $g \circ \varphi = \varphi^{\NN} \circ f$ and $Q \circ \varphi = \varphi \circ P$.

The comonad $\C'$ also induces an adjunction
$$\langle V'_{\C'}, G'_{\C'}, \eta'_{\C'}, \varepsilon'_{\C'} \rangle: \RBA_{\C'}
\rightharpoonup \RBA,$$
where
$$V'_{\C'} : \RBA_{\C'} \rightarrow \RBA$$ is given on objects
by $V'_{\C'}\langle (A,P),f\rangle = (A,P)$ and on morphisms
$\varphi:\langle (A,P),f\rangle\rightarrow\langle (B,Q),g\rangle$ in $\RBA_{\C'}$ by $V'_{\C'}(\varphi)=\varphi$.  The functor
$$G'_{\C'}:\RBA \rightarrow \RBA_{\C'}$$ is defined on objects
$(A,P) \in \RBA$ by $G'_{\C'}(A,P) = (A^{\NN}, \widetilde{P}),\delta_{(A,P)})$, and on morphisms
$\phi:(A,P)\rightarrow (B,Q)$ in $\RBA$ by $G'_{\C'}(\phi)= \phi^{\NN}$.  The natural transformations $\varepsilon'_{\C'}$ and
$\eta'_{\C'}$ are defined similarly to $\varepsilon'$ and $\eta'$.

Consequently there is a uniquely defined cocomparison functor
$H':\DRB \rightarrow \RBA_{\C'}$ such that $H' \circ G' = G'_{\C'}$ and $V'_{\C'} \circ H' = V'$.  Here $H'(R,d,P) = \langle (R,P),\tilde{d} \rangle $,
where for the $\lambda$-derivation
$d:R \rightarrow R$, the Rota-Baxter algebra homomorphism $\tilde{d}:R \rightarrow R^{\NN}$ is defined by
$\tilde{d} = V'(\eta'_{(R,d,P)})$.
Hence, for any
$a \in R$ and $n \in \NN$, $(\tilde{d}(a))(n) = d^{(n)}(a)$.

\begin{theorem}
The cocomparison functor $H':\DRB \rightarrow \RBA_{\C'}$ is
an isomorphism, i.e., $\DRB$ is comonadic over $\RBA$.
\mlabel{thm:Rcomonadic}
\end{theorem}

\begin{proof}
The proof, which uses the dual of Beck's Theorem~\cite{Ma} to show that $H'$ is an isomorphism, is virtually the dual of the proof of Theorem~\ref{thm:rbmon} and is omitted.
\end{proof}

In order to obtain the isomorphism $\DRB  \cong (\ALG^{\T})_{\widetilde{\C}}$, we need the following result from category theory. It is a part of the "folklore" of category theory. We include a proof for completeness.

\begin{lemma}
Suppose that $\A$ and $\B$ are categories, $K : \A \rightarrow \B$ is an isomorphism of categories, $\C = \langle C, \varepsilon, \delta \rangle $ is a comonad on $\A$ and $\C' = \langle C', \varepsilon', \delta' \rangle $ is a comonad on $\B$.
If $K$ commutes with $\C$ and $\C'$, i.e., $KC = C'K$, $K\varepsilon = \varepsilon' K$ and $K\delta = \delta'K$, then there exists a unique isomorphism $\widetilde{K}: \A_{\C} \rightarrow \B_{\C'}$ that lifts $K$, i.e., $U_{\C'}\widetilde{K} = KU_{\C}$.
\mlabel{lem:liftiso}
\end{lemma}

\begin{proof}
Since $K : \A \rightarrow \B$ is an isomorphism of categories, there is a functor
$H : \B \rightarrow \A$ such that $HK=\id_{\A}$ and $KH=\id_{\B}$ where $\id_{\A}$ and $\id_{\B}$ are the identity functors on $\A$ and $\B$, respectively.

For any $\langle a,f\rangle\in\A_{\C}$,
we claim that $\langle K(a),K(f)\rangle$ is in $\B_{\C'}$.
Since $KC = C'K$, we have $K(C(a))=C'(K(a))$ and $K(C(f))=C'(K(f))$. As $f$ is a morphism in $\A$ from $a$ to $C(a)$, $K(f)$ is a morphism in $\B$ from $K(a)$ to $K(C(a))=C'(K(a))$.
From $K\varepsilon = \varepsilon' K$ and $K\delta = \delta'K$, we get
$K(\varepsilon_a) = \varepsilon'_{K(a)}$ and $K(\delta_a) = \delta'_{K(a)}$.
By $\langle a,f\rangle\in\A_{\C}$, we obtain $\varepsilon_a\circ f=\id_a$ and
$C(f)\circ f=\delta_a\circ f.$
Then $K(\varepsilon_{a})\circ K(f)=\id_{K(a)}$
and $K(C(f))\circ K(f)=K(\delta_a)\circ K(f)$ hold.
That is, $\varepsilon'_{K(a)}\circ K(f)=\id_{K(a)}$
and $C'(K(f))\circ K(f)=\delta'_{K(a)}\circ K(f)$ hold.
Then we find that $\langle K(a),K(f)\rangle$ is in $\B_{\C'}.$

Define a functor $\widetilde{K} : \A_{\C} \rightarrow \B_{\C'}$ given on
objects $\langle a,f\rangle\in\A_{\C}$ by $\widetilde{K}\langle a,f\rangle=\langle K(a),K(f)\rangle$ and on morphisms $\phi:\langle a_1,f_1\rangle\rightarrow\langle a_2,f_2\rangle$ by
$\widetilde{K}(\phi)=K(\phi)$.
Similarly, we can define a functor $\widetilde{H} : \B_{\C'}\rightarrow \A_{\C}.$

From $HK=\id_{\A}$, for any $\langle a,f\rangle$ and the identity morphism $\id_{\langle a,f\rangle}$,
we get
$$(\widetilde{H}\widetilde{K})\langle a,f\rangle=\langle (HK)(a),(HK)(f)\rangle=\langle a,f\rangle$$
and
$$(\widetilde{H}\widetilde{K})(\id_{\langle a,f\rangle})=\widetilde{H}(\id_{\widetilde{K}\langle a,f\rangle})=\id_{(\widetilde{H}\widetilde{K})\langle a,f\rangle}=\id_{\langle a,f\rangle}.$$
For any morphisms $\phi:\langle a_1,f_1\rangle\rightarrow\langle a_2,f_2\rangle$
and $\varphi:\langle a_2,f_2\rangle\rightarrow\langle a_3,f_3\rangle$,
we get $$(\widetilde{H}\widetilde{K})(\varphi\phi)=\widetilde{H}(K(\varphi)K(\phi))
=(HK)(\varphi)(HK)(\phi)=\varphi\phi.$$
Then $\widetilde{H}\widetilde{K}=\id_{\A_{\C}}$
 where $\id_{\A_{\C}}$ is the identity functor on $\A_{\C}$.
Similarly, from $KH=\id_{\B}$, we can get  $\widetilde{K}\widetilde{H}=\id_{\B_{\C}}$
 where $\id_{\B_{\C}}$ is the identity functor on $\B_{\C}$.
That is, $\widetilde{K}: \A_{\C} \rightarrow \B_{\C'}$ is an isomorphism of categories.

The uniqueness of $\widetilde{K}$ follows because the functors $U_{\C}$ and $U_{\C'}$ are faithful.
\end{proof}

\begin{coro}
There an isomorphism of categories $\widetilde{K}: \RBA_{\C'} \rightarrow (\ALG^{\T})_{\widetilde{C}}$ such that
$V_{\widetilde{\C}}\widetilde{K}=K V'_{\C'}$.
\mlabel{coro:isolift}
\end{coro}

\begin{proof}
Just apply Lemma~\ref{lem:liftiso} to the case where $\A = \RBA$, $\B = \ALG^{\T}$, $K$ is the isomorphism from Theorem~\ref{thm:rbmon}, $\C'$ is the comonad on $\RBA$ that follows from Proposition~\ref{pp:adjDRB}, and $\widetilde{\C}$ is the comonad on $\ALG^{\T}$ from Corollary~\ref{coro:lifting}.
\end{proof}

We next consider the monad on differential algebras giving differential Rota-Baxter algebras. Observe that there is a monad $\T' = \langle T', \eta', \mu' \rangle$ on the category $\DIF$ given by $T' : \DIF \rightarrow \DIF$, where $T'(A,d) = (\sha(A), \widetilde{d})$. Here $\widetilde{d} : \sha(A) \rightarrow \sha(A)$ is defined as follows:
\begin{eqnarray*}
\lefteqn{ \widetilde{d}(x_0\otimes x_1\otimes\ldots\otimes x_n)}\\
&=&
d(x_0)\otimes x_1\otimes \ldots \otimes x_n + x_0x_1\otimes
x_2\ldots \otimes x_n +\lambda d(x_0) x_1\otimes x_1\otimes
\ldots \otimes x_n
\end{eqnarray*}
for $x_0\otimes \ldots \otimes x_n\in
A^{\otimes (n+1)}$ and then extending by linearity. Here we
use the convention that when $n=0$, $\widetilde{d}(x_0)=d(x_0)$.  Then we have the following result from\cite{GK3}:

\begin{theorem} $($\cite{GK3}$)$
Let $(A,d)$ be a
differential algebra of weight $\lambda$.
\begin{enumerate}
\item
The algebra embedding
\[ j_A: A \rar \sha(A)\]
is a morphism of differential algebras of weight $\lambda$.
\item
The quadruple $(\sha(A),\widetilde{d},P_A,j_A)$ is a free differential
Rota-Baxter algebra of weight $\lambda$ on the differential
algebra $(A,d)$, as described by the
following universal property: For any differential Rota-Baxter
algebra $(R,D,P)$ of weight $\lambda$ and any
differential algebra map $\varphi: (A,d)\rar (R,D)$, there
exists a unique differential Rota-Baxter algebra
homomorphism $\tilde{\varphi}:(\sha(A),\widetilde{d},P_A)\rar (R,D,P)$
such that $U'(\tilde{\varphi})\circ j_A=\varphi$.
\end{enumerate}
\mlabel{thm:diff}
\end{theorem}

Let $F':\DIF \rar\DRB$ denote the functor given on
objects $(A,d) \in \DIF$ by $F'(A,d) = (\sha(A),\widetilde{d},P_A)$ and on morphisms
$f:(A,d) \rar (A',d')$ in $\DIF$ by $$F'(f)(\sum^{k}_{i=1}a_{i0}\otimes a_{i1}\otimes\cdots \otimes a_{in_{i}})=\sum^{k}_{i=1}f(a_{i0})\otimes f(a_{i1})\otimes\cdots \otimes f(a_{in_{i}}),\quad
\sum^{k}_{i=1}a_{i0}\otimes a_{i1}\otimes\cdots \otimes a_{in_{i}}\in\sha(A).$$
There are two natural transformations $\eta': \id_{\DIF} \rightarrow U'F'$ and $\varepsilon': F'U' \rightarrow \id_{\DRB}$.
For any $(A,d) \in \DIF$, we define
$\eta'_{(A,d)}:(A,d)\rightarrow (\sha(A),\widetilde{d})$ to be just the natural embedding map.
For any $(R,D,P) \in \DRB$, define $\varepsilon'_{(R,D,P)}: (\sha(R),\widetilde{d},P_R) \rightarrow (R,D,P)$ by
$$\varepsilon'_{(R,d,P)}\Big(\sum^{k}_{i=1}a_{i0}\otimes a_{i1}\otimes\cdots \otimes a_{in_{i}}\Big) =
\sum^{k}_{i=1}a_{i0}P(a_{i1}P(\cdots P(a_{in_{i}})\cdots)),\quad \sum^{k}_{i=1}a_{i0}\otimes a_{i1}\otimes\cdots \otimes a_{in_{i}} \in \sha(R).$$

As a consequence of Theorem~\mref{thm:diff}, we obtain
\begin{coro}
The functor $F':\DIF \rightarrow \DRB$ defined above is the left adjoint of the forgetful functor
$U':\DRB \rightarrow \DIF$. Moreover, we get an adjunction $\langle F', U', \eta', \varepsilon' \rangle:\DIF\rightharpoonup\DRB$.
\end{coro}

The above adjunction $\langle F', U', \eta', \varepsilon' \rangle:\DIF\rightharpoonup\DRB$ gives rise to a monad $\T'=\langle T',\eta',\mu' \rangle$ on $\DIF$,  where
$T' = U' F':\DIF \rightarrow \DRB$
is a functor and $\mu':=U'\varepsilon' F': T' T'\rightarrow T'$.
Indeed, for any $(A,d) \in \DIF$, $T'(A,d)=(\sha(A),\widetilde{d})$ and
$\mu'_{(A,d)}:(\sha(\sha(A)),\widetilde{\widetilde{d}})\rightarrow (\sha(A), \widetilde{d})$ is extended by
 \begin{eqnarray*}
&&\mu'_{(A,d)}( (a_{10}\otimes\cdots\otimes a_{1n_{1}})\otimes\cdots\otimes(a_{k0}\otimes\cdots\otimes a_{kn_{k}}))\\
&=&\mu_A( (a_{10}\otimes\cdots\otimes a_{1n_{1}})\otimes\cdots\otimes(a_{k0}\otimes\cdots\otimes a_{kn_{k}}))\\
&=&(a_{10}\otimes\cdots\otimes a_{1n_{1}}) P_A(\cdots P_A(a_{k0}\otimes\cdots\otimes a_{kn_{k}})\cdots).
\end{eqnarray*}

As before, the monad $\T'$ induces a category
of $\T'$-algebras, denoted by $\DIF^{\T'}$ and gives rise to an adjunction
$$\langle F'^{\T'}, U'^{\T'},\eta'^{\T'} , \varepsilon'^{\T'} \rangle:\DIF
\rightharpoonup \DIF^{\T'},$$
where
$F'^{\T'} : \DIF \rightarrow \DIF^{\T'}$ is given on objects
by $F'^{\T'}(A,d) = \langle(\sha(A),\widetilde{d}),\mu'_{(A,d)}\rangle$ and on morphisms
$\varphi:(A,d)\rightarrow (A', d')$ in $\DIF$ by $F'^{\T'}(\varphi)=F'(\varphi)$.  The functor
$U'^{\T'}:\DIF^{\T'}\rightarrow \DIF$
is defined on objects
$\langle (A,d),f\rangle$ by $U'^{\T'}\langle (A,d),f\rangle = (A,d)$, and on morphisms
$\phi:\langle (A,d),f\rangle\rightarrow\langle (A',d'),f'\rangle$ in $\DIF^{\T'}$
by $U'^{\T'}(\phi)=\phi$.  The natural transformations $\varepsilon'^{\T'}$ and
$\eta'^{\T'}$ are defined similarly as $\varepsilon'$ and $\eta'$.
Then there is a uniquely defined comparison functor
$K':\DRB \rightarrow \DIF^{\T'}$ such that $K'\circ F'=F'^{\T'}$
and $U'^{\T'}\circ K'=U'$. Here $K'(R,D,P)=\langle (R,D),U'(\varepsilon'_{(R,D,P)})\rangle$ for any $(R,D,P)\in\DRB$.

\begin{theorem}
The comparison functor $K':\DRB \rightarrow \DIF^{\T'}$ is
an isomorphism, i.e., $\DRB$ is monadic over $\DIF$.
\mlabel{thm:drbmon}
\end{theorem}

\begin{proof}
The proof is similar to that of Theorem~\mref{thm:rbmon}.
\end{proof}

We will end the paper with the following diagram summarizing the constructions and main results of the paper.

\begin{theorem}
All the polygons in the following diagram commute.
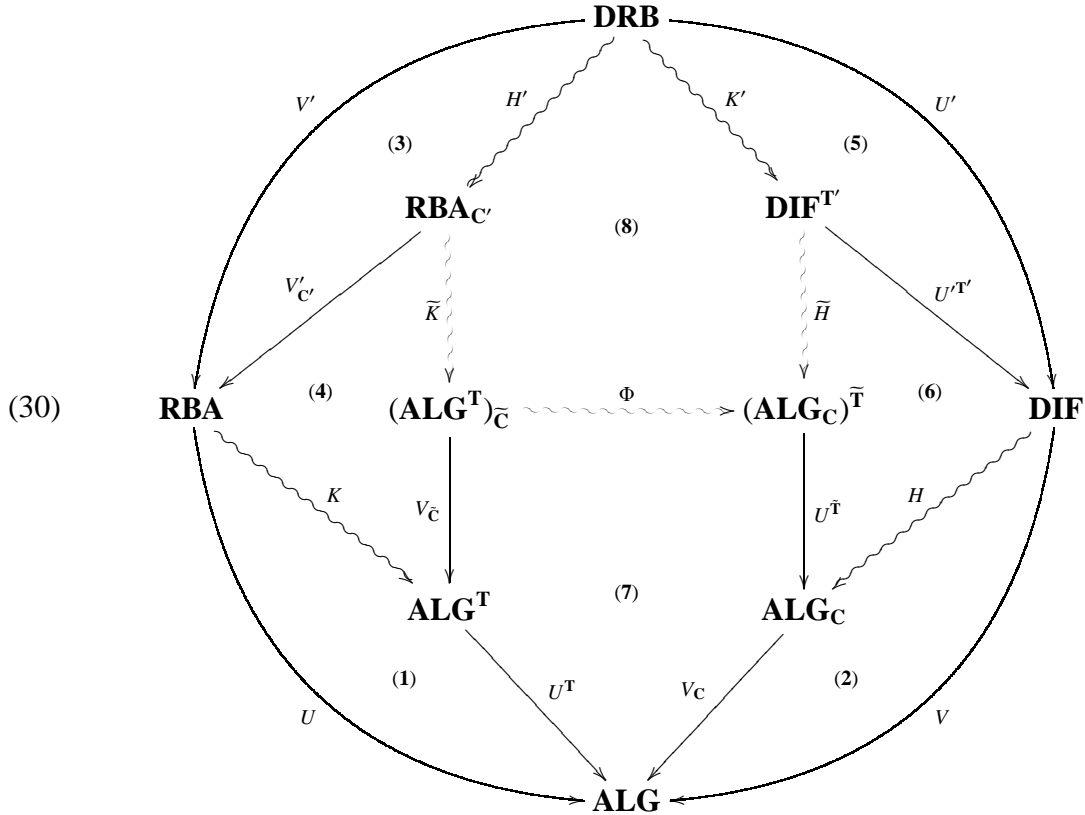
\begin{equation}
\begin{array}{c}
\xymatrix{
&&& \DRB \ar@{}[dddlll]_{\quad\mathbf{(3)}} \ar@/_4pc/_{V'}[ddddlll] \ar@{~>}_{H'}[ddl] \ar@{~>}^{K'}[ddr] \ar@{}[dddrrr]^{\mathbf{(5)}}\ar@/^4pc/^{U'}[ddddrrr] &&& \\
&&&&&& \\
&& \RBA_{\C'} \ar_{V'_{\C'}}[ddll]\ar@{~>}_{\widetilde{K}}[dd] \ar@{}_{\mathbf{(8)}}[rr]&&
\DIF^{\T'} \ar@{~>}^{\widetilde{H}}[dd] \ar^{U'^{\T'}}[rrdd] &&\\
&&&&&&\\
\RBA \ar@{~>}^{K}[rrdd] \ar@/_4pc/_{U}[rrrdddd] \ar@{}^{\mathbf{(4)}}[rr]&& (\ALG^{\T})_{\widetilde{\C}} \ar_{V_{\tilde{\C}}}[dd] \ar@{~>}^\Phi[rr] && (\ALG_\C)^{\widetilde{\T}}\ar^{U^{\tilde{\T}}}[dd] &&\DIF \ar@{~>}_{H}[lldd] \ar@/^4pc/^{V}[ddddlll] \ar@{}_{\mathbf{(6)}}[ll]\\
&&&&&&\\
&& \ALG^\T \ar^{U^{\T}}[ddr] \ar@{}^{\mathbf{(7)}}[rr] && \ALG_{\C}\ar_{V_{\C}}[ddl]  && \\
&&&\ar@{}^{\mathbf{(1)}}[lllu] \ar@{}_{\mathbf{(2)}}[rrru]& &&\\
&&& \ALG &&&
}
\end{array}
\mlabel{eq:diag}
\end{equation}
Here all of the functors with labels of $U$ or $V$ or any variation of those are forgetful functors.  All of the $U$ functors have a left adjoint, and all of the $V$ functors have a right adjoint, and all of them generate monads or comonads that make up this diagram. On the other hand, all of the functors labeled $H$ or $K$ are isomorphisms of categories and are denoted by wiggled arrows.
\mlabel{thm:diag}
\end{theorem}
\begin{proof}
For any $(R,P)\in\RBA$, we have
$(U^{\mathbf{T}}K)(R,P)=U^{\mathbf{T}}\langle R,U(\varepsilon_{(R,P)})\rangle
=R=U(R,P)$,
that is, subdiagram (1) commutes.
Similarly, we verify that subdiagrams~(2), (3) and (5)  commute.

From the definition of $\widetilde{K}$ in Lemma \ref{lem:liftiso} and Corollary \ref{coro:isolift}, we obtain the commutativity of subdiagram~(4). Similarly, we find that subdiagram~(6) commutes.

For any $\langle\langle A,h\rangle,f\rangle\in(\ALG^{\T})_{\widetilde{\C}}$, we have
$$(V_{\mathbf{C}}U^{\widetilde{\mathbf{T}}}\Phi)\langle\langle A,h\rangle,f\rangle
=(V_{\mathbf{C}}U^{\widetilde{\mathbf{T}}})\langle\langle A,f\rangle,h\rangle
=V_{\mathbf{C}}\langle A,f\rangle=A$$ and
$$(U^{\mathbf{T}}V_{\widetilde{\mathbf{C}}})\langle\langle A,h\rangle,f\rangle
=U^{\mathbf{T}}\langle A,h\rangle=A.$$ Hence $(V_{\mathbf{C}}U^{\widetilde{\mathbf{T}}}\Phi)\langle\langle A,h\rangle,f\rangle=(U^{\mathbf{T}}V_{\widetilde{\mathbf{C}}})\langle\langle A,h\rangle,f\rangle$. That is, $V_{\mathbf{C}}U^{\widetilde{\mathbf{T}}}\Phi=U^{\mathbf{T}}V_{\widetilde{\mathbf{C}}}$. Then subdiagram (7) commutes.

For $(R,D,P)\in\DRB$, we have
\begin{eqnarray*}(\Phi \widetilde{K} H')(R,D,P)&=&(\Phi \widetilde{K})\langle(R,P),V'(\eta'_{(R,D,P)})\rangle\\
&=&\Phi\langle\langle R,U(\varepsilon_{(R,P)})\rangle,K(V'(\eta'_{(R,D,P)}))\rangle\\&=&
\langle\langle R,K(V'(\eta'_{(R,D,P)}))\rangle,U(\varepsilon_{(R,P)})\rangle\end{eqnarray*}
and
$$(\widetilde{H}K')(R,D,P) =\widetilde{H}\langle(R,D),U'(\varepsilon'_{{(R,D,P)})})\rangle
=\langle\langle R,V(\eta_{(R,D)})\rangle,H(U'(\varepsilon'_{(R,D,P)}))\rangle.$$
From $$K(V'(\eta'_{(R,D,P)}))=V(\eta_{(R,D)})$$ and $$U(\varepsilon_{(R,P)})=H(U'(\varepsilon'_{(R,D,P)})),$$ we get
$(\Phi \widetilde{K} H')(R,D,P)=(\widetilde{H}K')(R,D,P)$. That is, $\Phi \widetilde{K} H'=\widetilde{H}K'$.
Then subdiagram (8) commutes.
\end{proof}

It is important to note that the isomorphism $\Phi :   (\ALG^{\T})_{\widetilde{\C}} \rightarrow (\ALG_{\C})^{\widetilde{\T}}$ ties this diagram together, and that this isomorphism comes directly from the mixed distributive law of $\T$ over $\C$.
\smallskip

\noindent
{\bf Acknowledgements}:
This work is supported by the National Natural Science Foundation of China (Grant No. 11371178) and the National Science Foundation of US (Grant No. DMS~1001855).

\addcontentsline{toc}{section}{\numberline {}References}

\end{document}